\documentclass{amsart}
\usepackage[colorlinks, linkcolor=blue,  citecolor=blue, urlcolor=blue]{hyperref}%
\usepackage{amsthm}
\usepackage{url}
\usepackage{xspace}
\usepackage{graphicx}
\usepackage{subfig} 
\usepackage{enumerate}
\usepackage{verbatim}
\usepackage{alltt}
\usepackage{color}
\usepackage[alphabetic]{amsrefs}
\graphicspath{{../figures/}}

\newtheorem{theorem}{Theorem}
\theoremstyle{remark}
\newtheorem{remark}{Remark}
\newtheorem{definition}{Definition}

\newtheorem{lemma}{Lemma}
\numberwithin{equation}{section}

\newcommand{\norm}[1]{\Vert#1\Vert}

\newcommand{\abs}[1]{\vert#1\vert}
\newcommand{\Abs}[1]{\left\vert#1\right\vert}

\newcommand{\grad}{\nabla}
\newcommand{\bq}{\begin{equation}}
\newcommand{\eq}{\end{equation}}
\newcommand{\R}{\mathbb{R}}
\newcommand{\Rd}{\R^d}

\newcommand{\e}{\epsilon}
\newcommand{\bO}{\mathcal{O}}

\newcommand{\FM}{F^{\e}_M}
\newcommand{\FN}{F_N^\e}
\newcommand{\FA}{F_A^\e}
\newcommand{\Fe}{F^\e}
\newcommand{\MA}{{Monge-Amp\`ere}\xspace}

\newcommand{\USC}{\text{USC}}
\newcommand{\LSC}{\text{LSC}}
\newcommand{\ex}[1]{\times 10^{- #1}}
\newcommand{\G}{\mathcal{G}}
\newcommand{\Dt}{\mathcal{D}}

\newcommand{\xv}{\mathbf{x}}
\newcommand{\xo}{\mathbf{x}_0}

\newcommand{\www}{.25}

\begin{document}

\title[Filtered Schemes for Monge-Amp\`ere]
{Convergent  filtered schemes for the Monge-Amp\`ere partial differential equation
}
\author{Brittany D. Froese}
\thanks{Department of Mathematics, University of Texas at Austin,  1 University Station C1200, Austin, TX, 78712 ({\tt bfroese@math.utexas.edu})}

\author{Adam M. Oberman}
\thanks{Department of Mathematics and Statistics, McGill University, 805 Sherbrooke Street West, Montreal, Quebec, H3A 0G4, Canada ({\tt adam.oberman@mcgill.ca})}

\begin{abstract}
The theory of viscosity solutions has been effective for representing and approximating weak solutions to fully nonlinear Partial Differential Equations (PDEs) such as the elliptic \MA equation.  The approximation theory of Barles-Souganidis~\cite{BSnum} requires that numerical schemes be monotone (or elliptic in the sense of~\cite{ObermanDiffSchemes}).  But such schemes have limited accuracy.  
In this article, we establish a convergence result for filtered schemes, which are nearly monotone.
This allows us to construct finite difference discretizations of arbitrarily high-order.  We demonstrate that the higher accuracy is achieved when solutions are sufficiently smooth.
In addition, the filtered scheme provides a natural detection principle for singularities.  We employ this framework to construct a formally second-order scheme for the \MA equation and present computational results on smooth and singular solutions.
\end{abstract}

\date{\today}



\maketitle

\tableofcontents

\section{Introduction}\label{sec:intro}

The numerical approximation of the \MA equation is a problem of current interest, {because of} the many applications of the equation to various fields, and because the equation is the prototypical fully nonlinear elliptic Partial Differential Equation (PDE).  Thus building effective (convergent, fast, accurate) solvers to this equation {demonstrates the possibility} of effectively solving a wide class of fully nonlinear PDEs, which until recently were believed to be intractable.

We consider the \MA equation in a  convex {bounded} subset~$ \Omega \subset \Rd$ 
\bq
\label{MA}\tag{MA}
\det(D^2u(x)) = f(x), \quad \text{for $x$ in }\Omega,
\eq
where  $\det(D^2u)$,  is the determinant of the Hessian of the function $u$.   
We include Dirichlet boundary conditions on the boundary~$\partial\Omega$,
\bq\tag{D}\label{Dirichlet}
u(x) = g(x), \quad \text{for $x$ on }\partial\Omega.
\eq
{This} equation is augmented by the convexity constraint
\bq\label{convex}\tag{C}
u \text{ is convex, }
\eq
which is necessary for the equation to be elliptic.

This article builds on a series of papers which have developed solution methods for the {Monge-Amp\`ere} equation.  The foundation of schemes for a wide class of nonlinear parabolic and elliptic equations was developed in~\cite{ObermanDiffSchemes}.  The first convergent scheme for the \MA equation was built in~\cite{ObermanEigenvalues}; this was restricted to two dimensions and to a slow iterative solver.   Implicit solution methods were first developed in~\cite{BenamouFroeseObermanMA}, where it was demonstrated that the use of non-monotone schemes led to slow solvers for singular solutions.   In~\cite{ObermanFroeseMATheory} a higher dimensional monotone discretization was constructed, a fast Newton solver was also implemented.  The convergent discretizations use a wide stencil scheme, which leads to a reduction on accuracy which reflects the directional resolution of the stencil.  While this cannot be avoided on singular solutions, it is desirable to have a more accurate solver on (rare) smooth solutions.
 In~\cite{ObermanFroeseFast}  a hybrid solver was built, which combined the advantages of accuracy in smooth regions, and robustness (convergence and stability) near singularities.   However, this was accomplished at the expense of a convergence proof.   While the Dirichlet problem is a natural starting point, for applications related to mapping problems or Optimal Transportation, other boundary conditions are used.  These boundary conditions were implemented in~\cite{FroeseTransport}.   In a work in progress, the filtered scheme has also been applied to the Optimal Transportation boundary conditions~\cite{BenamouFroeseObermanOT}.
  
\subsection{Contribution of this work}\label{sec:contribution}
In this article, we improve on previous results by building a convergent, higher order accurate scheme for the \MA equation, combining the advantages of the hybrid scheme in~\cite{ObermanFroeseFast} with the convergence proof in~\cite{ObermanFroeseMATheory}.   Our result requires that we extend the convergence theory of Barles and Souganidis by considering the more general class of nearly monotone schemes.  This extension is of independent interest, and applies to schemes for a wide class of elliptic equations.  Once this theoretical result is established, it leads to a natural and simple method for constructing accurate discretizations of {the Monge-Amp\`ere} equation, and indeed for the entire class of nonlinear elliptic PDEs, given the foundation of a monotone elliptic scheme.

The  combined schemes are called  \emph{filtered} finite difference {approximations}.  We provide a proof that solutions of the filtered scheme exist and converge to the viscosity solution of the underlying PDE.  
 The theory ensures, and computational results verify, that solutions of this scheme converge to the viscosity solution of the \MA equation even in the singular setting.  Using Newton's method, we construct a fast solver for the resulting nonlinear system of equations.

The advantage of the more {accurate} scheme is obvious when solutions are smooth. 
In the case of a singular solution, the added accuracy of the filtered scheme is redundant, at least near singular parts of the solution.  However, the since the low accuracy of the monotone scheme is not a problem on singular solutions, 
the filtered scheme allows for the use of a narrow stencil in general, while still achieving full accuracy on smooth solutions.

\subsection{The heuristics of a hybrid scheme}
A natural way to build hybrid schemes is to make a convex combination of a stable, convergent scheme, $F_M$, and an accurate, less stable scheme, $F_A$.  The weighting of the convex combination is determined by a function which measures the regularity of the solution so that
\[
F_H = w_s F_M + (1-w_s) F_A, 
\]
where $0\le w_s \le 1$ is a continuous function which is 1 in a neighbourhood of a singularity and goes to 0 elsewhere.  In general $w_s$ could be determined by checking the size of a derivative of $u$, for example $\norm{D^2u} \ge 1/h$, where $h$ is the grid spacing. 
We record the idea of a generic hybrid scheme with the following schematic
\bq\label{hybrid}
F_H = 
\begin{cases}
F_M & \text{ near singularity, e.g. $\norm{D^2u} \ge 1/h$}, \\
F_A & \text{ elsewhere }.
\end{cases} 
\eq

  So the function $w_s$ depends on the solution, and this could potentially lead to instabilities unrelated to the stability of the underlying schemes $F_M$ and $F_A$.  

In~\cite{ObermanFroeseFast}, $F_M$ corresponded to a wide stencil, monotone elliptic scheme, and $F_A$ corresponded to the standard nine point finite difference scheme.
{We review those schemes below}.  In that article, we were able to use the regularity theory for the Dirichlet problem for~\eqref{MA} to determine $w_s$ \emph{a priori}.   
However, in other problems of interest, for example with Optimal Transportation boundary conditions, $w_s$ will depend on the solution.
Since the resulting hybrid scheme was not monotone, we could not apply the Barles-Souganidis theorem and there was no other obvious way to prove convergence.
However the method worked well in practice.  

To summarize: hybrid schemes are practical tools which blend the accuracy and stability of the underlying schemes.  But they are defined in an \emph{ad hoc} manner, and there is no guarantee of the observed stability and accuracy.

\subsection{A motivating example for a filtered scheme}
In this section, we present a filtered scheme on a simpler equation,  to present the ideas which follow more clearly.

The filtered scheme provided an intrinsic  method for defining based on the size of the difference between the monotone operator and the accurate operator.  For illustration purposes, consider the model equation 
\[
F[u](x) = \abs{u_x} -1
\]
on the domain $[-1,1]$ with boundary conditions $u(-1) = u(1) =1$.  Then the viscosity solution is simply $u(x) = \abs{x}$.  Define the monotone upwind scheme
\[
F_M^h[u](x) = \max \left\{ 
\frac{u(x+h) - u(x)}{h},
\frac{u(x-h) - u(x)}{h}
\right\} -1,
\]
which is first order accurate: $F_M^h(\phi) - F(\phi) = \bO(h)$, for smooth $\phi$.
The scheme $F^M$ is consistent with the method of characteristics, and it satisfies a maximum principle, so solutions converge.  
The equation has the explicit form
\[
u(x) = \max \{ u(x+h),  u(x- h)\} - h.
\]
Next define the second order accurate, but unstable, centred difference scheme
\[
F_A^h = \frac{\abs{u(x+h) - u(x-h)}}{2h} -1,
\]
$F_A^h(\phi) - F(\phi) = \bO(h^2)$, for smooth $\phi$.

A natural definition of a singularity of the equation is when $\abs{u_{xx}}$ is large.  For the finite difference scheme, this can be interpreted as  
\[
\frac{\abs{u(x+h) -2u(x) + u(x-h) }}{h^2} \ge \frac 1 h.
\]
Taking the difference between the two schemes, we obtain a similar condition, up to the scaling in $h$,
\[
\abs{F_A^h - F_M^h} = \frac 1 2 \frac{\abs{u(x+h) -2u(x) + u(x-h) }}{h}.
\]
So a singularity can be defined by the condition
\bq\label{DefSing}
\Abs{\frac{F_A^h - F_M^h}{h}} \ge \frac 1 {h}.
\eq
The inequality above leads to a hard condition.  Instead we replace it {with} a soft condition using a continuous filter function. 
\begin{definition}[Filter function]\label{defn:filter}
We define a  filter function to {be} a continuous, bounded function, $S$, which is equal to the identity in a neighbourhood of the origin and zero outside.  
\end{definition}
\begin{figure}[hbt]
\includegraphics[trim=1.6in 1.9in 2in 2.5in, clip=true]
{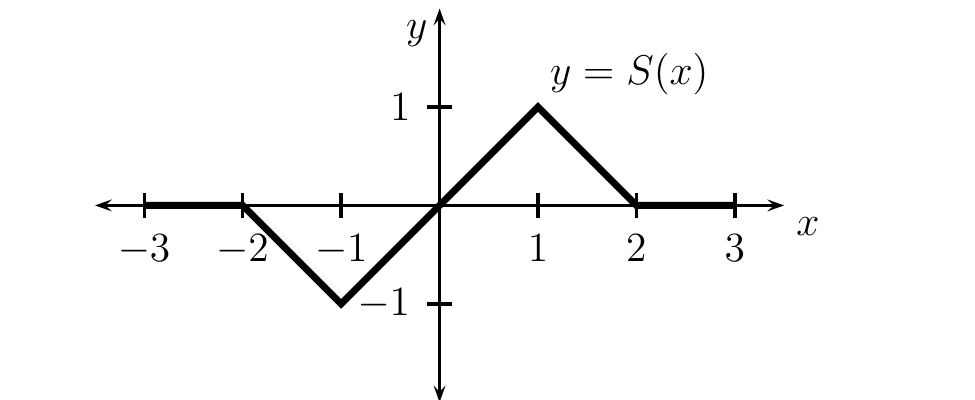}
\caption{
Filter function
}
\label{fig:filter}
\end{figure}
For example, 
\bq\label{eq:filter1}
S(x) = \begin{cases}
x & \abs{x} \leq 1 \\
0 & \abs{x} \ge 2\\
-x+ 2  & 1\le x \le 2 \\
-x-2  & -2\le x\le -1.
\end{cases} 
\eq
See \autoref{fig:filter}.

Then the filtered scheme can be defined as follows
\[
F_S = F_M + h S\left( \frac{F_A - F_M} {h} \right ).
\]
By virtue of the condition~\eqref{DefSing}, the filtered scheme will be consistent with the condition for a hybrid scheme~\eqref{hybrid}.  (In general, for~\eqref{DefSing} to hold, the scaling $h$ may depend on the order of accuracy of the schemes $F_A$ and $F_M$).

By the definition~\eqref{eq:filter1},
\[
\abs{F_S - F_M} \le 2h.
\]
{Thus} the difference between the filtered scheme, $F_S$, and the monotone scheme, $F_M$, goes to zero with $h$, uniformly over all functions.  This property is what allows us to prove convergence, since the filtered scheme is nearly monotone.

\subsection{Introduction to numerical methods for degenerate elliptic PDEs}\label{sec:introNumerics}
There are two major challenges in building numerical solvers for nonlinear and degenerate elliptic PDEs.  The first challenge is to build convergent approximations, often with finite difference schemes.  The second challenge is to build efficient solvers.

The approximation theory developed by Barles and Souganidis~\cite{BSnum}  provides criteria for the convergence of approximation schemes:  schemes that are consistent, monotone, and stable converge to the unique viscosity solution of a degenerate elliptic equation.  However, this work does not indicate how to build such schemes, or how to produce fast solvers for the schemes.  It is not obvious how to ensure that schemes satisfy the required comparison principle.  The class of schemes for which this property holds was identified in~\cite{ObermanDiffSchemes}, and were called \emph{degenerate elliptic}, by analogy with the structure condition for the PDE.     

An important distinction for this class of equations is between first order (Hamilton-Jacobi) equations and the second order (nonlinear elliptic) case.  The theory of viscosity solutions~\cite{CIL} covers both cases, but the corresponding numerical methods are quite different. In the first order case, where the method of characteristics is available, there are formulas for exact solutions (e.g. Hopf-Lax) and there is a connection with one-dimensional conservation laws~\cite{EvansBook}.  The second order case has more in common with divergence-structure elliptic equations; however, for degenerate of fully nonlinear equations, many of the tools from the divergence-structure case (e.g. finite elements, multi grid solvers) have not been successfully applied.

Much of the progress on discretization techniques and fast solvers is limited to the first order case.
For Hamilton-Jacobi equations, which are first order nonlinear PDEs, monotonicity is necessary for convergence. Early numerical papers studied explicit schemes for time-dependent equations on uniform grids~\cite{CrandallLionsNum, SougNum}.   These schemes have been extended to higher accuracy schemes,  which include second order convergent methods---the central schemes~\cite{LinTadmorHJ}, as well as higher order interpolation methods---the ENO schemes~\cite{OsherShuENO}.  Semi-Lagrangian schemes take advantage of the method of characteristics to prove convergence~\cite{FalconeSemiLagrangian}.  These have been extended to the case of differential games~\cite{FalconeBardiGames}. 
Two classes of fast solvers have been developed: fast marching~\cite{FastMarching, TsitsiklisFastMarching} and  fast sweeping~\cite{FastSweeping}. The fast marching and fast sweeping methods give fast solution methods for first order equations by taking advantage of the method of characteristics, which is not available in the second order case.  

There is much less work available for second order degenerate elliptic equations.  
For uniformly elliptic PDEs, monotone schemes are not always necessary for convergence (for example, most higher order finite element methods are not monotone).  
However, for fully nonlinear or degenerate elliptic equations, the only convergence proof currently available requires that schemes be monotone.   One way to ensure monotonicity is to use wide stencil finite difference schemes; this has been done for the equation for motion by mean curvature~\cite{ObermanMC}, for the Infinity Laplace equation~\cite{ObermanIL}, for functions of the eigenvalues~\cite{ObermanEigenvalues}, for Hamilton-Jacobi-Bellman equations~\cite{ZidaniWideHJBLong}, and for the convex envelope~\cite{ObermanCENumerics}.  
Even for linear elliptic equations, a wide stencil may be necessary to build a monotone scheme~\cite{MotzkinWasow}.


\subsection{Higher order numerical methods}\label{sec:introHO}
Monotone schemes have limited accuracy. For first-order equations,  a consistent, monotone scheme will be at most first-order accurate.  For second-order equations, the accuracy is at most second-order~\cite{ObermanDiffSchemes}.  Moreover, as we have already noted, wide stencils are needed to build monotone schemes for some equations.  In this case, the formal accuracy of the scheme depends not only on the spatial resolution of the scheme, but also on the angular resolution of the scheme.  To make large computations practical, it is desirable to restrict schemes to a reasonably narrow stencil width.  However, this can place severe restrictions on the accuracy that can be achieved.

For first-order Hamilton-Jacobi equations, some progress has been made in the construction of convergent higher-order schemes. Lions and Souganidis~\cite{LionsSougMUSCL} considered using a general higher-order scheme, which must be ``filtered'' to ensure the preservation of fractional one-sided second derivative bounds. These bounds, which are used instead of a monotonicity condition, coincide with the condition needed to ensure uniqueness of almost everywhere solutions of Hamilton-Jacobi equations; this approach does not generalize to second-order equations.  Using an approach more similar to the one presented in this article, Abgrall~\cite{Abgrall} proposed a ``blending'' of a monotone and a higher-order scheme.  Provided a solution to the blended scheme exists, it converges to the viscosity solution of the steady Hamilton-Jacobi equation.

In the case of second-order fully nonlinear elliptic equations, we are not aware of any convergent higher-order schemes.

\subsection{Numerical methods for the \MA equation}\label{sec:introMA}

For singular solutions, Newton's method combined with a non-monotone discretization of \MA can become unstable~\cite{ObermanFroeseFast}, which necessitates the use of other solvers for the discrete system of nonlinear equations.  
The simplest solver is to simply iterate the discrete version of the parabolic equation
\[u_t = \det(D^2u) - f\]
using forward Euler.   This method is slow, since the convergence time depends on the nonlinear CFL condition, which become more stringent with increasing problem size~\cite{BenamouFroeseObermanMA}.
Semi-implicit solvers can improve the solution time for smooth solutions, but can slow down or break down on singular solutions~\cite{BenamouFroeseObermanMA}.   Similar behavior was observed using other discretizations~
\cite{DGnum2006}.

 However, in addition to suffering from a severe time-step restriction, this approach does not enforce the convexity constraint.  In particular, this type of iteration is unstable for solutions that are non-strictly convex.  In two dimensions, we constructed solution methods that selected the convex solution.  While no convergence proof was available, the methods appeared to converge to the correct weak solution of the equation.  However, convergence was very slow for non-smooth or non-strictly convex solutions.

In order to ensure convergence, more sophisticated techniques are needed.  An early work by Oliker and Prussner~\cite{olikerprussner88} described a method that converges to the Aleksandrov solution in two dimensions.  This early method was used to solve a problem with only about a dozen grid points.  
Several other methods have been proposed in recent years~\cite{BrennerNeilanMA2D,DGnum2006,FengFully,LoeperMA}; these are similar to our standard finite difference methods in terms of the lack of convergence theory and behaviour on singular solutions.  We draw particular attention to the vanishing moment method proposed by Feng and Neilan~\cite{FengFully}, which involves regularizing the equation by adding a small multiple of the bilaplacian.  This approach will be discussed in more detail in~\S\ref{sec:bilaplace}.

\subsection{The finite difference discretization of \MA}\label{sec:MAmon}
We describe the  elliptic (monotone) representation of the \MA operator used in \cite{ObermanFroeseMATheory,FroeseTransport}.  

We begin with some informal remarks to explain the representation of the operator.  It bears repeating that the standard finite difference discretization is not elliptic or convex. 

One difficulty comes from the off diagonal terms $u_{xy}$.  So by using wide stencils and rotating the coordinate system we can hope find the coordinate system in which the Hessian matrix is diagonal.  But finding this coordinate system must be done in a monotone manner.   The way to do this is to use a version of Hadamard's inequality, which we interpret to give an expression for the determinant of a positive definite matrix as a minimization.   

Hadamard's inequality states that for a  positive semidefinite $d\times d$ matrix $M$,
\[
\det(M) \le \prod_{i=1}^d m_{ii}
\]
We can force equality by choosing a coordinate system where $M$ is diagonal.  We write this as
\[
\det(M) = \min_{N \in O} \prod_{i=1}^d  n_i^T M n_i, \qquad  \text{for $M$ positive definite}
\]
where $O$ is the set of all orthogonal $d\times d$ matrices, and $n_i, i = 1,\dots d$ are the rows of $N$.

Another difficulty comes from the fact that the operator is elliptic only when $M$ is positive definite.  So we would like a continuous extension of the operator to non-positive definite matrices.   A first possibility is
\[
\overline
\det'(M) = \min_{N \in O} \prod_{i=1}^d  \max \left\{ n_i^T M n_i, 0 \right\}\
\]
which extends the formula.  However, we would also like to apply Newton's method, which computes (and inverts) the gradient of the operator.  But this operator has a zero gradient.   So instead, we use the following extension, which gives a negative value with a non-zero gradient on (most) non-positive definite matrices.
\[
\overline\det(M) = \min_{N \in O} \left \{  \prod_{i=1}^d  \max \left\{n_i^T M n_i, 0 \right\} +  \sum_{i=1}^d  \min \left\{n_i^T M n_i, 0 \right\}  \right \}
\]
where 
\[
\overline\det(M) = 
\begin{cases}
\det(M), & \text{$M$ positive semi-definite }\\
\lambda_1 + \dots + \lambda_k, & \lambda_1 \le \dots \le \lambda_k < 0 \text{ are the negative eigenvalues of $M$ }
\end{cases}
\]
We now apply the representation to the matrix $D^2u(x)$.
The term $n_i^T D^2u n_i$ has a natural interpretation as a directional second derivatives, in the direction 
$n_i$.  This allows us to write
\[  
 \overline\det(D^2u) \equiv \min\limits_{N \in O}  \left \{  \max \left\{ \frac{ d^2 u}{d n_i^2}, 0 \right\} +  \sum_{i=1}^d  \min \left\{ \frac{ d^2 u}{d n_i^2}, 0 \right\}  \right \}
\]
where $O$ is the set of all orthogonal $d\times d$ matrices, and $n_i, i = 1,\dots d$ are the rows of $N$. 

The next stage is to replace the derivatives with finite differences, to arrive at a semi-discrete representation. 
However, we have only a limited number of directions available on the grid.
So we limit ourselves to considering a finite number of vectors $n_i$ that lie on the grid and have a fixed maximum length; this is the directional discretization, which gives us the angular resolution $d\theta$ of our stencil; see \autoref{fig:stencil}.  In this figure, values on the boundary are used to maintain the directional resolution $d\theta$ (at the expense of lower order accuracy in space because the distances from the reference point are not equal).  Another option is to use narrower stencils as the boundary is approached, which leads to lower angular resolution, but better spatial resolution.  We denote the resulting set of orthogonal vectors by $\G$.

\begin{figure}
  \centering
  \subfloat{
  \includegraphics[height=0.4\textwidth]
  {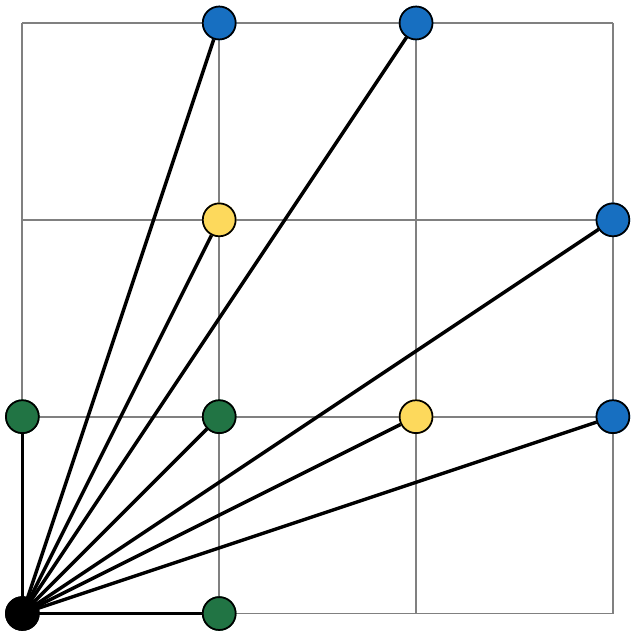}
  }                
    \subfloat{
  \includegraphics[height=0.4\textwidth]
  {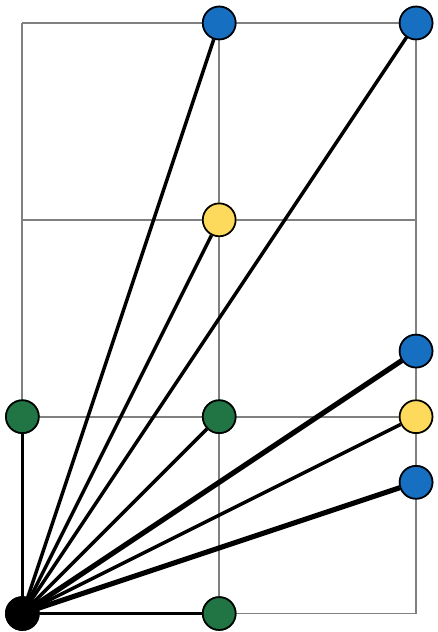}
  }    
  \subfloat{
  \includegraphics[height=0.4\textwidth]
  {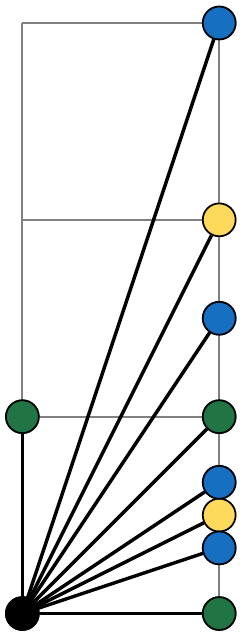}
  }
\caption{
Neighboring grid points used for width one  (green), two (yellow), and three (blue) stencils.  The illustration shows the neighbors in the first quadrant.  The modification near the boundary is illustrated in the second and third figures.
}
\label{fig:stencil}
\end{figure}

Each of the directional derivatives in the \MA operator is then discretized using centered differences:
\[ \Dt_{\nu\nu}u_i = \frac{1}{\abs{\nu}^2h^2}\left(u(x_i + \nu h) + u(x_i - \nu h) - 2u(x_i)\right). \]
Thus the discrete version of the \MA operator is
\[ 
\min\limits_{\{\nu_1\ldots\nu_d\}\in \G} \left\{
\prod\limits_{j=1}^{d} 
\max\{\Dt_{\nu_j\nu_j}u_i,\delta\} + \sum\limits_{j=1}^d\min\{\Dt_{\nu_j\nu_j}u_i,\delta\}\right\} 
 \]
where $\delta>0$ is a small parameter used to bound the second directional derivatives away from zero.

\section{Viscosity Solutions of Degenerate Elliptic Equations}\label{sec:viscosity}

\subsection{Background from viscosity solutions}
In this section we review the theory of viscosity solutions. See the references above and~\cite{CIL}.

Given an open domain $\Omega \subset \Rd$, a second order partial differential equation on $\Omega$ is a function
\[
F: \Omega \times \R \times \Rd \times \R^{d\times d} \to \R
\]
which we write as $F(x,r,p,M)$.  Given a function $u\in C^2(\Omega)$ we let $p = \grad u$ and $M = D^2u$ denote the gradient and Hessian of $u$, respectively.  
We include Dirichlet boundary conditions into the operator $F$ in order to pose equation~\eqref{PDE} in the closed domain $\bar{\Omega}$ by defining
\[
F[u](x) = u(x) - g(x), \quad x\in \partial \Omega.
\]
Then the function $u\in C^2(\Omega)$ is a solution of the PDE $F$ in $\Omega$ if 
\bq\label{PDE}\tag{PDE}
F[u](x) = F(x,u(x),\grad u(x), D^2u(x)) = 0, \quad x\in {\Omega}.
\eq
However, most of the PDEs we consider fail to have classical solutions under general assumptions on the data.  This motivates the definition of \emph{viscosity} solutions.
\begin{definition}
The equation~\eqref{PDE} is \emph{degenerate elliptic} if 
\[ F(x,r,p,X) \leq F(x,s,p,Y) \]
for all $x\in\bar{\Omega}$, $r,s\in\R$, $p\in\R^n$, $X,Y\in S^n$ with $X \geq Y$ and $r \leq s$.
Here $X \geq Y$ means that $X-Y$ is a positive definite matrix.
\end{definition}
The definition is related to the \emph{comparison principle}.
\begin{lemma}
The function $F(x,r,p,X)$ is degenerate elliptic if and only if whenever $x$ is a non-negative local maximum of $u-v$, for $u,v \in C^2$, $F[u](x) \geq F[v](x)$.
\end{lemma}
\begin{proof}If $x$ is a local maximum,  
$u \geq v$, $Dv = Du$, and $D^2u \leq D^2v,$  at  $x$.  Then $F(x,u,Du,D^2u) = F(x,u,{Dv}, D^2u) \geq F(x,{v},Dv, D^2u)
\geq F(x,v,Dv,{D^2v})$.
Here we have used the definition degenerate elliptic.
\end{proof}

Elliptic equations need not have smooth solutions, which necessitates some notion of a weak solution.  We are interested in the viscosity solution~\cite{CIL}.  Before we define this, we introduce the upper and lower semi-continuous envelopes of a function.

\begin{definition}[Upper and Lower Semi-Continuous Envelopes]\label{def:envelope}
The \emph{upper and lower semi-continuous envelopes} of a function $u(x)$ are defined, respectively, by
\[ u^*(x) = \limsup_{y\to x}u(y), \]
\[ u_*(x) = \liminf_{y\to x}u(y). \]
\end{definition}

\begin{definition}[Viscosity Solution]\label{def:viscosity}
An upper (lower) semi-continuous function $u$ is a \emph{viscosity sub(super)-solution} of~\eqref{PDE} if for every $\phi\in C^2(\bar{\Omega})$, whenever $u-\phi$ has a local maximum (minimum) at $x \in \bar{\Omega}$, then
\[ 
F_*(x,u(x),\nabla \phi(x),D^2\phi(x)) \leq   0 
\]
\[ (F^*(x,u(x),\nabla \phi(x),D^2\phi(x)) \geq   0). \]
A function $u$ is a \emph{viscosity solution} if it is both a sub- and a super-solution.  See \autoref{fig:viscositydefn}.
\end{definition}

\begin{figure}[hbt]
\includegraphics[height = .15\textheight]{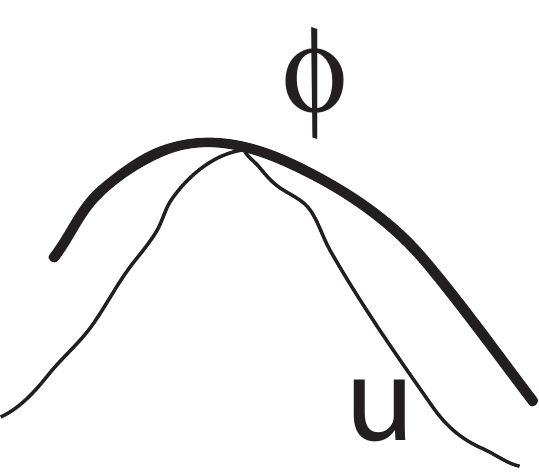}
\caption{Touching $u(x)$ by a smooth test function $\phi$.}
\label{fig:viscositydefn}
\end{figure}

\begin{remark}
When checking the definition of a viscosity solution, we can limit ourselves to considering unique, strict, global maxima (minima) of $u-\phi$ with a value of zero at the extremum.  See, for example,~\cite[Prop. 2.2]{KoikeViscosity}.  For an accessible introduction to viscosity solution in the first order case, see~\cite[Chapter 10]{EvansBook}.
\end{remark}

The equations we consider satisfy a \emph{comparison principle}.  If $u\in\USC(\bar{\Omega})$ is a sub-solution and $v\in\LSC(\bar{\Omega})$ is a super-solution of~\eqref{PDE}, then $u\leq v$ on $\bar{\Omega}$.  The proof of this result is one of the main technical arguments in the viscosity solutions theory~\cite{CIL}.  For a general proof and also simplified proofs in special cases, we refer to~\cite[Chapter 3]{KoikeViscosity}.

\subsection{Approximation schemes}\label{sec:approxSchemes} 

An approximation scheme is family of functions parameterized by $\e \in \R^+$, 
\[
\Fe:\Omega \times \R \times L^\infty(\Omega) \to \R
\]
which we write as $\Fe(x,r,u(\cdot))$.   Given a function $v\in L^\infty(\Omega)$, we write 
\bq\label{eq:approx} 
\Fe[v](x) = \Fe(x,v(x),v(\cdot)) 
\eq
where $\e>0$ is small.   The function $u^\e$ is a solution of the scheme $\Fe$ if
\bq\label{PDEe}  \tag*{$(\textnormal{PDE})^\e$}
\Fe[u^\e](x) = 0, \quad \text{ for all $x$ in } \Omega
\eq

\begin{remark}
For finite difference schemes, which are defined on a grid, piecewise linear interpolation is monotone, so we can extend the function continuously onto the domain.  We assume in the current setting that we are working with the extended function.
\end{remark}

The important properties of a scheme follow.

\begin{definition}[Consistent]\label{def:consistent}
The scheme~\eqref{eq:approx} is \emph{consistent} with the equation~\eqref{PDE} if for any smooth function $\phi$ and $x\in\bar{\Omega}$,
\[ \limsup_{\e\to0,y\to x,\xi\to0} \Fe(y,\phi(y)+\xi,\phi(\cdot)+\xi) \leq F^*(x,\phi(x),\nabla\phi(x),D^2\phi(x)), 
\]
\[ \liminf_{\e\to0,y\to x,\xi\to0} \Fe(y,\phi(y)+\xi,\phi(\cdot)+\xi) \geq F_*(x,\phi(x),\nabla\phi(x),D^2\phi(x)). \]
\end{definition}

Accuracy is defined inside the domain.  To be more precise, we could also define accuracy at the boundary.
\begin{definition}[Accurate]\label{def:accurate}
The scheme~\eqref{eq:approx} is \emph{accurate} to $\bO(\e^\alpha)$
 if for any smooth function $\phi$ and $x\in {\Omega}$,
\[
\Fe[\phi] - F[\phi] = \bO(\e^\alpha).
\]
\end{definition}

\begin{definition}[Stable]\label{def:stable}
The scheme~\eqref{eq:approx} is \emph{stable} if any solution $u^\e$ of~\eqref{eq:approx} is bounded independently of $\e$.
\end{definition}

\begin{definition}[Monotone]\label{def:Mon}
The scheme~\eqref{eq:approx} is \emph{monotone} if for every $\e>0$, $x\in\bar{\Omega}$, $s\in\R$ and bounded $u,v $,
\[ 
u \geq v  \implies
\Fe(x,s,u(\cdot)) \leq \Fe(x,s,v(\cdot)). 
 \]
\end{definition}

\begin{definition}[Elliptic]\label{def:Ell}
The scheme~\eqref{eq:approx} is \emph{elliptic} if it can be written
\[
\Fe[v] = \Fe(x, v(x), v(x) - v(\cdot)), 
\]
where $\Fe$ is nondecreasing in its second and third arguments,
\bq\label{felliptic}
s \le t,~ u(\cdot) \le v(\cdot) \implies  \Fe(x,s,u(\cdot)) \le  \Fe(x,t,v(\cdot))
\eq
\end{definition}

Elliptic schemes are monotone, since
\begin{align*}
u(\cdot) \ge v(\cdot) 
& \implies s - u(\cdot) \le s - v(\cdot)
\\
&  \implies  \Fe(x, s, s - u(\cdot)) \le \Fe(x, s, s - v(\cdot)).
\end{align*}

In addition, it is shown in~\cite{ObermanDiffSchemes} that under mild technical conditions  solutions of elliptic schemes exist, satisfy a comparison principle, and are stable.
These technical conditions are easily satisfied by finite difference schemes.
For example if we consider the Dirichlet problem, a small perturbation of the scheme (which can go to zero with $\e$) can be introduced to ensure that these conditions are satisfied.  A more complete theory for well-posedness of elliptic schemes can be found in \cite{ObermanGames}, which does not require perturbations.

\begin{definition}
The scheme $\FN$ is a \emph{perturbation} if there is a nonnegative modulus function $m:\R^+\to \R^+$ with 
\[
\lim_{\e\to 0^+} m(\e) = 0
\]
 such that
\[
\sup_{u \in L^\infty(\Omega)} \sup_{x\in \Omega} \Abs{\FN[u](x)} \le m(\e).
\]
\end{definition}

\begin{definition}[Nearly Monotone]\label{def:nearMon}
The scheme~$\Fe$ is \emph{nearly monotone} if it can be written as
\bq\label{nearMon} 
\Fe[u] =  \FM[u] + \FN[u] 
\eq
where $\FM$ is monotone and $\FN$ is a perturbation.
\end{definition}

%
%

\subsection{Convergence proof}
Theorem~\ref{thm:converge}, generalizes the corresponding result for monotone schemes in~\cite{BSnum}.  
It uses Lemma~\ref{lem:sequences} below, which is a standard result from viscosity solutions theory. We include the proof 
 since a reference to the result is not easily found.

%
%

\begin{theorem}[Convergence of Approximation Schemes]\label{thm:converge}
Let $u$ be the unique solution of the degenerate elliptic~\eqref{PDE}.  
For each $\e>0$,  let $u^\e$ be a stable solution of the nearly elliptic approximation scheme 
~\ref{PDEe}.  Then 
\[
u^\e \to u, \quad \text{ locally uniformly,  as } \e \to 0
\]\end{theorem}

\begin{proof}[Proof of Theorem~\ref{thm:converge}]
Define
\[ \bar{u}(x) = \limsup_{\e\to0,y\to x} u^\e(y) \in USC(\bar{\Omega}), \]
\[ \underline{u}(x) = \liminf_{\e\to0,y\to x} u^\e(y) \in LSC(\bar{\Omega}). \]
These functions are bounded by the stability of solutions (Definition~\ref{def:stable}). 
We record the fact that 
\bq \label{ineq1}
\underline{u} \leq \bar{u}. 
\eq
If we know that $\bar{u}$ is a subsolution and $\underline{u}$ is a supersolution, 
then we could apply the comparison principle for~\eqref{PDE} to $\bar{u}$ and $\underline{u}$ to conclude that
\[ 
\bar{u} \leq \underline{u}. 
\]
Together these inequalities imply that $\bar{u} = \underline{u}$.  The local uniform convergence follows from the definitions.

It remains to show that $\bar{u}$ is a subsolution and $\underline{u}$ is a supersolution.  

We proceed to show that $\bar{u}$ is a sub-solution by testing the condition in definition~\ref{def:viscosity}.  Given a smooth test function $\phi$, let $x_0$ be a strict global maximum of $\bar{u}-\phi$ with $\phi(x_0) = \bar{u}(x_0)$.   

 By Lemma~\ref{lem:sequences} (below), we can find sequences with 
\[
\begin{cases}
\e_n \to 0\\
y_n \to x_0\\
u^{\e_n}(y_n) \to \bar{u}(x_0).
\end{cases}
\]
where $y_n$ is a global maximizer of $u^{\e_n}-\phi$.

Define
\bq\label{xi}
\xi_n = u^{\e_n}(y_n)-\phi(y_n).
\eq
Then
\[
\xi_n \to \bar{u}(x_0) - \phi(x_0) = 0 
\]
and
$u^{\e_n}(x) - \phi(x) \leq u^{\e_n}(y_n) - \phi(y_n) = \xi_n$  for any  $x\in\bar{\Omega}.$ 
In particular
\bq\label{ulessv}
u^{\e_n}(\cdot) - \phi(\cdot) \leq \xi_n.
\eq

As a consequence of Definitions~\ref{def:Mon} and \ref{def:nearMon} we see that
\bq\label{nearlymon}
u(\cdot)  \le v(\cdot) 
 \implies 
\Fe(x,s,u(\cdot)) \ge \Fe(x,s,v(\cdot)) - 2 m(\e),
\eq
for every $\e>0$, $x\in\bar{\Omega}$, $s\in\R$.


Using the definitions above and \eqref{nearlymon}, we find that
\begin{align*}
0 &= F^{\e_n}[u^{\e_n}](y_n) = F^{\e_n}(y_n,u^{\e_n}(y_n),u^{\e_n}(\cdot)), 
&& \text{ since $u^{\e_n}$ is a solution},
\\
  &= F^{\e_n}(y_n, \phi(y_n) + \xi_n, \phi(\cdot) + (u^{\e_n}(\cdot)-\phi(\cdot)))
  && \text{ by \eqref{xi}},
  \\
  &\geq F^{\e_n}(y_n,\phi(y_n)+\xi_n,\phi(\cdot)+\xi_n) - 2m(\e)
  && \text{ by \eqref{ulessv} applied to  \eqref{nearlymon}}.
\end{align*}
Next, we compute
\begin{align*}
0 &\geq \liminf_{n\to\infty}\left\{F^{\e_n}(y_n,\phi(y_n)+\xi_n,\phi(\cdot)+\xi_n)- 2m(\e) \right\}\\
  &\geq \liminf_{\e\to0,y\to x,\xi\to0}F^{\e}(y,\phi(y)+\xi,\phi(\cdot)+\xi)\\
  &= F_*(x_0, \phi(x_0),\nabla\phi(x_0),D^2\phi(x_0))
&& \text{by definition~\ref{def:consistent}}
  \\
  &= F_*(x_0, \bar{u}(x_0),\nabla\phi(x_0),D^2\phi(x_0)),
\end{align*}
which shows that $\bar{u}$ is a subsolution.  

By a similar argument,  we can show that $\underline{u}$ is a super-solution.  
\end{proof}

\begin{lemma}[Stability of Maxima]\label{lem:sequences}
Suppose the family $u^\e$ is bounded uniformly in $\e$.
Define
\[ \bar{u}(x) = \limsup_{\e\to0,y\to x} u^\e(y) \in USC(\bar{\Omega}), \]
Given a smooth function $\phi$, let $x_0$ be a strict global maximum of $\bar{u}-\phi$.  Then there exist sequences:
\[
\begin{cases}
\e_n \to 0\\
y_n \to x_0\\
u^{\e_n}(y_n) \to \bar{u}(x_0)
\end{cases}
\]
where $y_n$ is a global maximizer of $u^{\e_n}-\phi$.
\end{lemma}


\begin{proof}[Proof of Lemma~\ref{lem:sequences}]
From the definition of the limit superior, we can find sequences
\[ \e_n \to 0, \quad z_n \to x_0 \]
such that
\[ u^{\e_n}(z_n) \to \bar{u}(x_0). \]
Now we define $y_n \in \bar{\Omega}$ to be maximizers of 
$u^{\e_n}(x)-\phi(x). $

We have
\[ u^{\e_n}(y_n)-\phi(y_n) \geq u^{\e_n}(z_n) - \phi(z_n) \to \bar{u}(x_0) - \phi(x_0) = 0. \]
Also, for any $\delta>0$ and large enough $n$,
\[ u^{\e_n}(y_n)-\phi(y_n) \leq \bar{u}(y_n)-\phi(y_n) + \delta \leq \bar{u}(x_0)-\phi(x_0) + \delta = \delta. \]
Thus we have
\[ u^{\e_n}(y_n)-\phi(y_n) \to 0. \]

Now suppose we do not have $y_n \to x_0$.  Then (possibly through a subsequence) there is an $R >0$ such that
\[ \abs{y_n - x_0} > R. \]
Also, since the maximum is strict, global, and unique, there is a $K>0$ such that 
\[ \bar{u}(y) - \phi(y) < -K <0 \]
whenever $\abs{y-x_0} > R$.

Thus for any $\delta>0$ and large enough $n$,
\[ u^{\e_n}(y_n) - \phi(y_n) \leq \bar{u}(y_n) -\phi(y_n) + \delta < -K + \delta \to -K < 0, \]
which contradicts the fact that $u^{\e_n}(y_n)-\phi(y_n) \to 0$.  We conclude that
\[ y_n \to x_0. \]

Finally, it is clear that
\begin{align*}
\abs{u^{\e_n}(y_n) - \bar{u}(x_0)} &= \abs{u^{\e_n}(y_n)-\phi(x_0)}\\
  &\leq \abs{u^{\e_n}(y_n)-\phi(y_n)} + \abs{\phi(y_n)-\phi(x_0)} \\
  &\to 0.
\end{align*}
Therefore,
\[ u^{\e_n}(y_n) \to \bar{u}(x_0). \qedhere \]
\end{proof}

\subsection{Solutions  of nearly monotone finite difference methods}\label{sec:nearlyMonotone}

%
%
%
%

\autoref{thm:converge} assumed existence and stability of solutions to nearly monotone schemes.  We show next that this follows from  well-posedness of the underlying monotone schemes.

For simplicity, we work with grid functions, i.e. we assume that $\Omega$ is a finite set, for example a finite difference grid, which is identified with $\R^N$.   For general approximation schemes, some other form of compactness can be used to achieve the same results.

\begin{lemma}[Existence and Stability of nearly monotone schemes]\label{lem:exist}
Suppose that solutions exist and are stable for the inhomogeneous problem for monotone scheme~
\[
\FM[u] + g = 0
\]
If $\FN$ is a continuous perturbation, then stable solutions exist for the nearly monotone scheme~\eqref{nearMon}.
\end{lemma}

\begin{proof}
\newcommand{\Su}{\mathcal{S}}
\emph{Existence}
Fix $\e>0$.  Write $u = \Su(g)$ for the solution operator 
  of the scheme
\[ 
\FM[u] + g= 0. 
\]
By assumption, $\FN$ is uniformly  bounded on its domain,  
\[
\FN(u)  \le R, \quad \text{ for all } u.
\]
Since $\Su$ is continuous, we have
\[
\norm {\Su(\FN) } \le R_2.
\]
In particular, for the ball $B_{R_2} \subset \R^N$, 
\[
\Su(\FN(B_{R_2})) \subset B_{R_2}.
\]
Applying Brouwer's fixed point theorem, we concluded that there is a fixed point $\Su(\FN(u^*)) = u^*$, which means
\[
\FM[u^*] + \FN[u^*] = 0.
\]
Thus we have identified a solution of the nearly monotone scheme.

\emph{Stability} 
We wish to show that any solution $u^\e$ of the nearly monotone scheme
can be bounded uniformly as $\e\to0$.  Then $u^\e$ is a solution of 
\[
\FM[v] + \FN[u^\e] =0,
\]
regarding $v$ as the unknown. Since $\FN$ is bounded, continuity of the solution operator for $\FM$ gives a uniform bound for $v$.

\end{proof}

\section{Applications}

\subsection{Filtered schemes}\label{sec:higherOrder}
We will now construct a more accurate nearly monotone scheme, starting from the monotone scheme $\FM$ and the consistent (and more accurate) scheme $\FA$.

This is accomplished by setting 
\bq\label{filteredscheme}
\FN[u] =  \e^\alpha S\left (\frac{\FA[u]-\FM[u]}{\e^\alpha}\right)
\eq
in~\eqref{nearMon},where $S(x)$ is a filter function, as defined in Definition~\ref{defn:filter}.
Since 
$
\norm{\FN} = \e^\alpha
$, ~\eqref{nearMon} is a nearly monotone scheme.

\begin{lemma}[Filtered scheme is accurate]
Suppose that the formal discretization errors of the schemes $F_A$, $F_M$ are $\bO(\e^{\beta_A})$ and $\bO(\e^{\beta_M})$ respectively.  Choose the parameter $\alpha$ so that $\beta_A>\beta_M>\alpha>0$.  
If $\phi$ is smooth, then $F^{\e}[\phi]= \FA[\phi].$
\end{lemma}
\begin{proof}
If $\phi$ is smooth, then 
\begin{align*}
\frac{\FA[\phi]-\FM[\phi]}{\e^\alpha} = \frac{\bO(\e^{\beta_A})+\bO(\e^{\beta_M})}{\e^\alpha}
  = \bO(\e^{\beta_M-\alpha})
  < \bO(1).
\end{align*}
Using~\eqref{filteredscheme}, we find that $\FN[\phi] = \FA[\phi] - \FM[\phi]$, which means $F^{\e}[\phi]= \FA[\phi].$
\end{proof}

\begin{remark}[Singularity detection]
The filtered scheme provides an intrinsic singularity detection mechanism, which is adapted to the discretization (and the equation) itself.  This means fewer points are considered singular, compared to simply using a regularity condition.  The selection principle is illustrated in \autoref{fig:filters}.
\end{remark}

\subsection{Perturbations}
Given the degenerate elliptic~\eqref{PDE}, 
and a continuous operator $G[u]$, we can consider the nearly monotone approximation
\[
\Fe[u] = F[u] +  \e C S(G[u]/C),
\]
for some large constant $C$.
Theorem~\ref{thm:converge} gives a convergence proof.
We consider the special case of a solution $u$ of $F[u]$, where $G[u]$ is bounded by $C$.  Then for $\e$ small enough we recover the unfiltered perturbation.

For numerical methods, if we know \emph{a priori} that the solution satisfies the bounds above, we can use an accurate method. However, for weak solutions of the PDE, this is not the case.

\subsection{Bilaplacian regularization }\label{sec:bilaplace}

The Vanishing Moment Method proposed by Feng and Neilan~\cite{FengFully} for solving the second-order elliptic equation~\eqref{PDE}
where we specify the boundary conditions
\[
u = g, \quad \text{ on } \partial \Omega
\]
by perturbing with the bilaplacian 
\bq\label{eq:vm}  F[u] + \e \Delta^2 u = 0. \eq
This also introduces additional boundary conditions, for example
\[
\Delta u = 0, \text{ on } \partial \Omega. 
\]
Here we provide a proof of convergence of the filtered PDE in the case of smooth solutions.

Consider instead the filtered PDE
\bq\label{eq:vm_fil} \Fe[u] \equiv F[u] + \e C S\left(\frac{\Delta^2u}{C}\right) \eq
along with filtered boundary conditions
\bq
\e S\left(\frac{\Delta u}{C}\right) = 0, \quad \text{ on } \partial \Omega,
\eq
for a large constant $C$.  
The scheme we defined is nearly monotone.
If $\Abs{{\Delta^2u}}, \abs{\Delta u} \le C$ then we recover~\eqref{eq:vm}.  
This follows since if the unperturbed equation is smooth, adding $\e \Delta^2 u$ only improves regularity in the interior.

Theorem~\ref{thm:converge} then provides a convergence proof for the method in the case of smooth solutions.

\section{Computational Results}\label{sec:compute}
In this section we implement the filtered scheme \eqref{filteredscheme} for the \MA equation, using a monotone scheme and standard finite differences.

\subsection{Numerical implementation}\label{sec:implement}
In this section, we implement a convergent filtered finite difference scheme for the \MA equation.  The purpose of this section is to demonstrate that the filtered scheme:
\begin{enumerate}
\item Allows for higher accuracy {than} the monotone scheme.
\item Can be solved efficiently using Newton's method.
\end{enumerate}

In practice, we do not expect the results to be much better than for the hybrid scheme.  The improvement is that the hybrid scheme was overly conservative, so the filtered scheme reduces to the accurate scheme.  But the real advantage of this method is that it can be used in situations (such as the Optimal Transportation problem) where the hybrid scheme may not converge. 

We now solve the \MA equation using the filtered scheme
\[ F_H^{h,d\theta}[u] = F_M^{h,d\theta}[u] + \e(h,d\theta) S\left(\frac{F_A^h[u]-F_M^{h,d\theta}[u]}{\e(h,d\theta)}\right). \]
Here  $F_M$ is the monotone scheme described in~\autoref{sec:MAmon} and $F_A$ is the formally second order accurate scheme  obtained using a standard centred-difference discretisation (see~\cite{ObermanFroeseFast} for details).  
In addition, we have regularized the $\max$ operator to allow the gradient to be computed, as in~\cite{ObermanFroeseFast}.
In the computations below, we used  $\e(h,d\theta) = \sqrt{h} + d\theta/10$.

We solve the resulting system of nonlinear equations using Newton's method:
\[ u^{n+1} = u^n - (\nabla F_H[u^n])^{-1}F_H[u^n] \]
where the Jacobian will be given by
\[ \nabla F_H[u] = \left(1-S'[u]\right)\nabla F_M[u] + S'[u]\nabla F_A[u]. \]

The derivative of the filter~\eqref{eq:filter1} is given by
\[
S'(x) = \begin{cases}
1 & \abs{x} < 1\\
-1 & 1 < \abs{x} < 2\\
0 & \abs{x}>2.
\end{cases}
\]
However, allowing this derivative to take on negative values can lead to poorly conditioned or ill-posed linear systems.  Instead, we approximate the Jacobian by
\[ \tilde{\nabla} F_H[u] = \left(1-S'[u]\right)\nabla F_M[u] + \max\{S'[u],0\}\nabla F_A[u]. \]

To properly assess the speed and accuracy of our filtered method, we also solve the \MA equation using
\begin{enumerate}
\item The monotone scheme described in~\autoref{sec:MAmon}.
\item The \emph{a priori} hybrid scheme presented in~\cite{ObermanFroeseFast} (which has no convergence proof).
\item The formally second-order standard scheme solved using one of the methods of~\cite{BenamouFroeseObermanMA}.
\end{enumerate}

\subsection{Computational examples}\label{sec:examples}
We use our method to compute the following four representative examples described in~\cite{ObermanFroeseMATheory}.  

Throughout these definitions, we use $\xv = (x,y)$ to denote a general point in $\R^2$ and $\xo = (0.5,0.5)$ for the center of the domain.

Given a solution $u(x,y)$, we obtain the mapping $\grad u(x,y)$ from the computational square to the image of the square.  We visualize the optimal mapping\autoref{fig:solns} by plotting the images of the constant $x$ and constant $y$ lines from the square on the image set.

The first example, which is smooth and radial, is given by
\bq\label{eq:c2} 
u(\xv) = \exp \left( \frac{ \norm{\xv-\xo}^2}{2} \right),
\qquad 
f(\xv) = \left(1+ \norm{\xv-\xo}^2\right)\exp\left( \norm{\xv-\xo}^2 \right).
\eq
The second example, which is $C^1$, is given by
\bq\label{eq:c1} 
u(\xv) = \frac{1}{2}\left( (\norm{\xv-\xv_0} -0.2)^+\right )^2, 
\quad
f(\xv) = 
\left( 
1 - \frac{0.2}{\norm{\xv-\xv_0}}
\right)^+.
\eq

The third example is smooth in the interior of the domain, but has an unbounded gradient near the boundary point $(1,1)$.  The solution is given by
\bq\label{eq:blowup} 
u(\xv) = -\sqrt{2-\norm{\xv}^2},
\qquad 
f(\xv) = 2 {\left(2-\norm{\xv}^2\right)^{-2} }.
\eq

The final example is the cone, which is only Lipschitz continuous.
\bq\label{eq:cone} 
u(\xv) ={\norm{\xv-\xv_0}},
\qquad
f = \mu = \pi \,\delta_{\xv_0}.
\eq
In fact, this solution is not even a viscosity solution; it must be understood using the more general notion of an Aleksandrov solution.
In order to approximate the solution on a grid with spatial resolution $h$ using viscosity solutions, we approximate the measure $\mu$ by its average over the ball of radius $h/2$, which gives
\[ 
f^h  = 
\begin{cases}
4/h^2 &  \text{ for } \norm{\xv - \xv_0} \leq h/2,\\
0 & \text{ otherwise.}
\end{cases}
\]
The solutions and the corresponding mappings are plotted in \autoref{fig:solns}.

\newcommand{\ww}{.35}
\begin{figure}[htdp]
	\centering
       \subfloat[]{
       \label{fa}
       \hspace{-.2\textwidth}
       \includegraphics[width=\ww\textwidth]{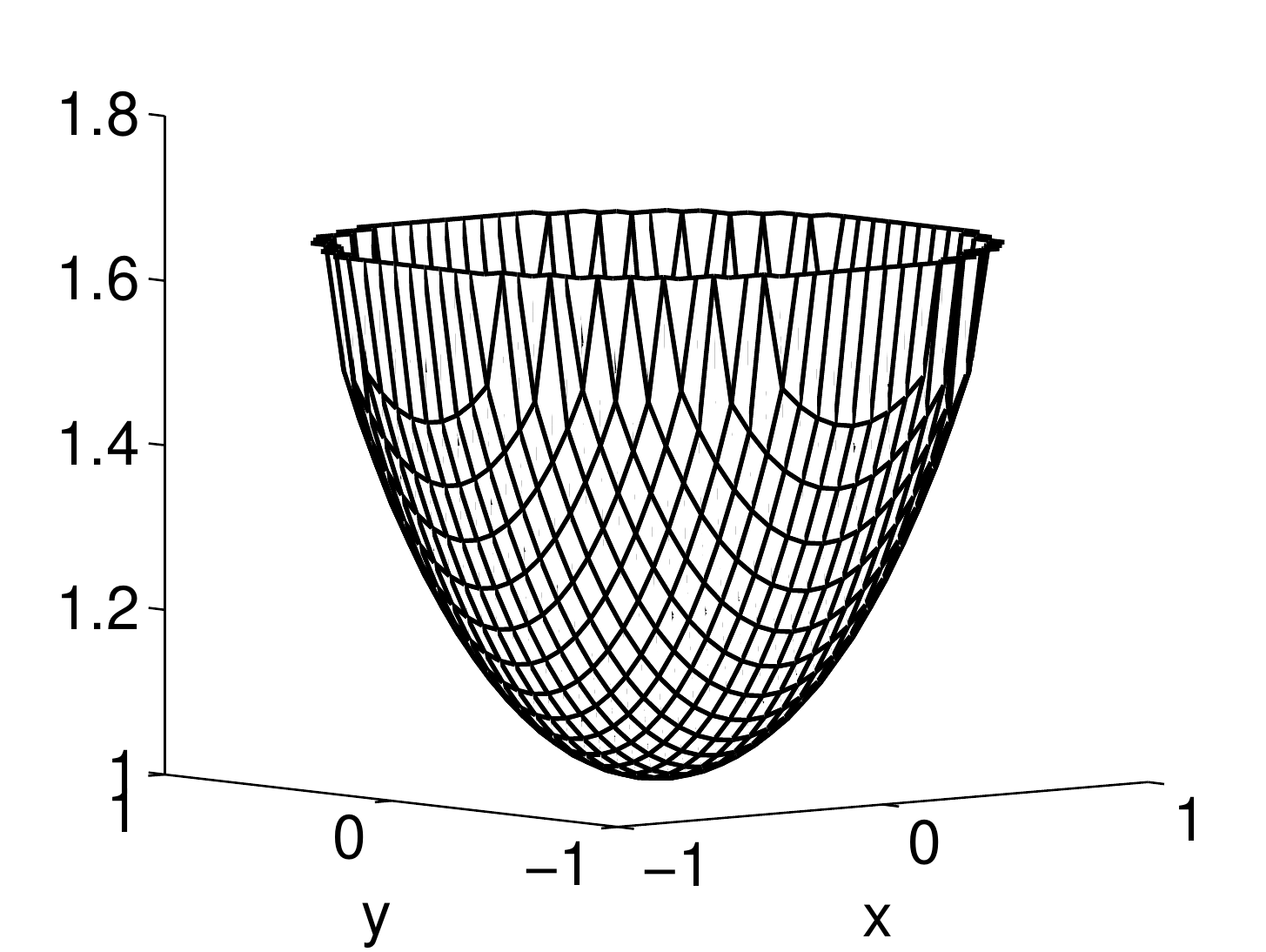}  
       \includegraphics[width=\ww\textwidth]{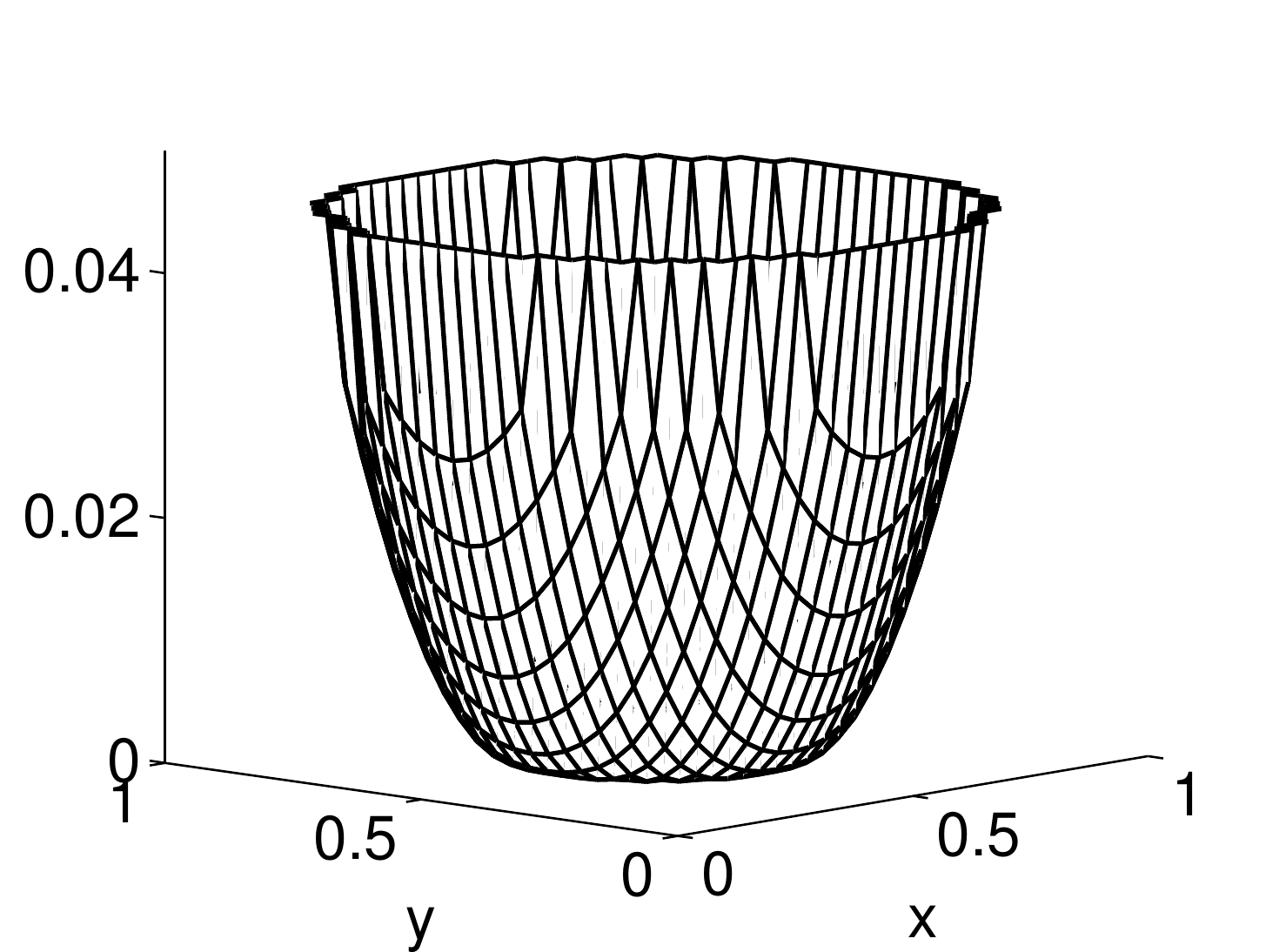}
       \includegraphics[width=\ww\textwidth]{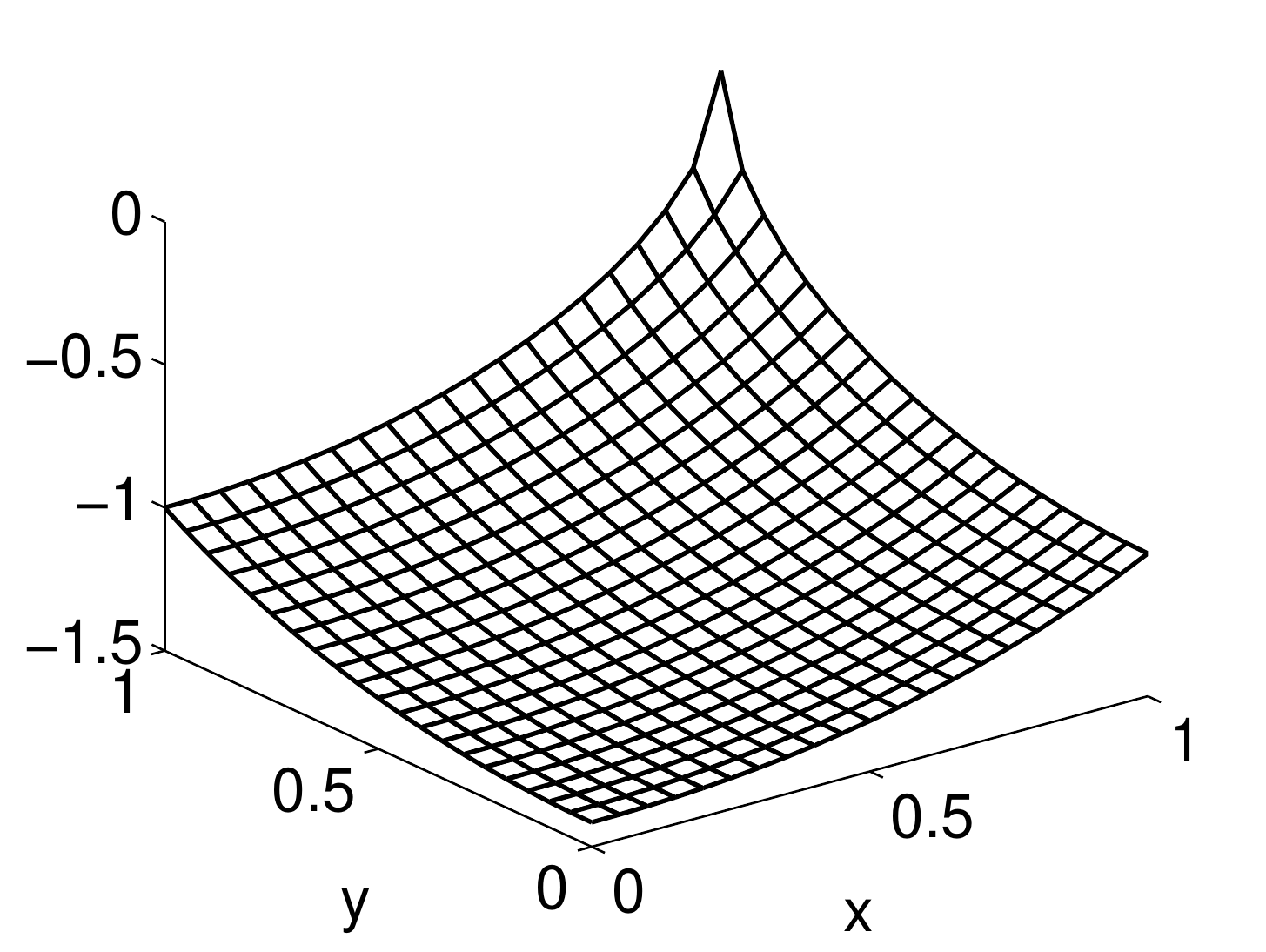}
       \includegraphics[width=\ww\textwidth]{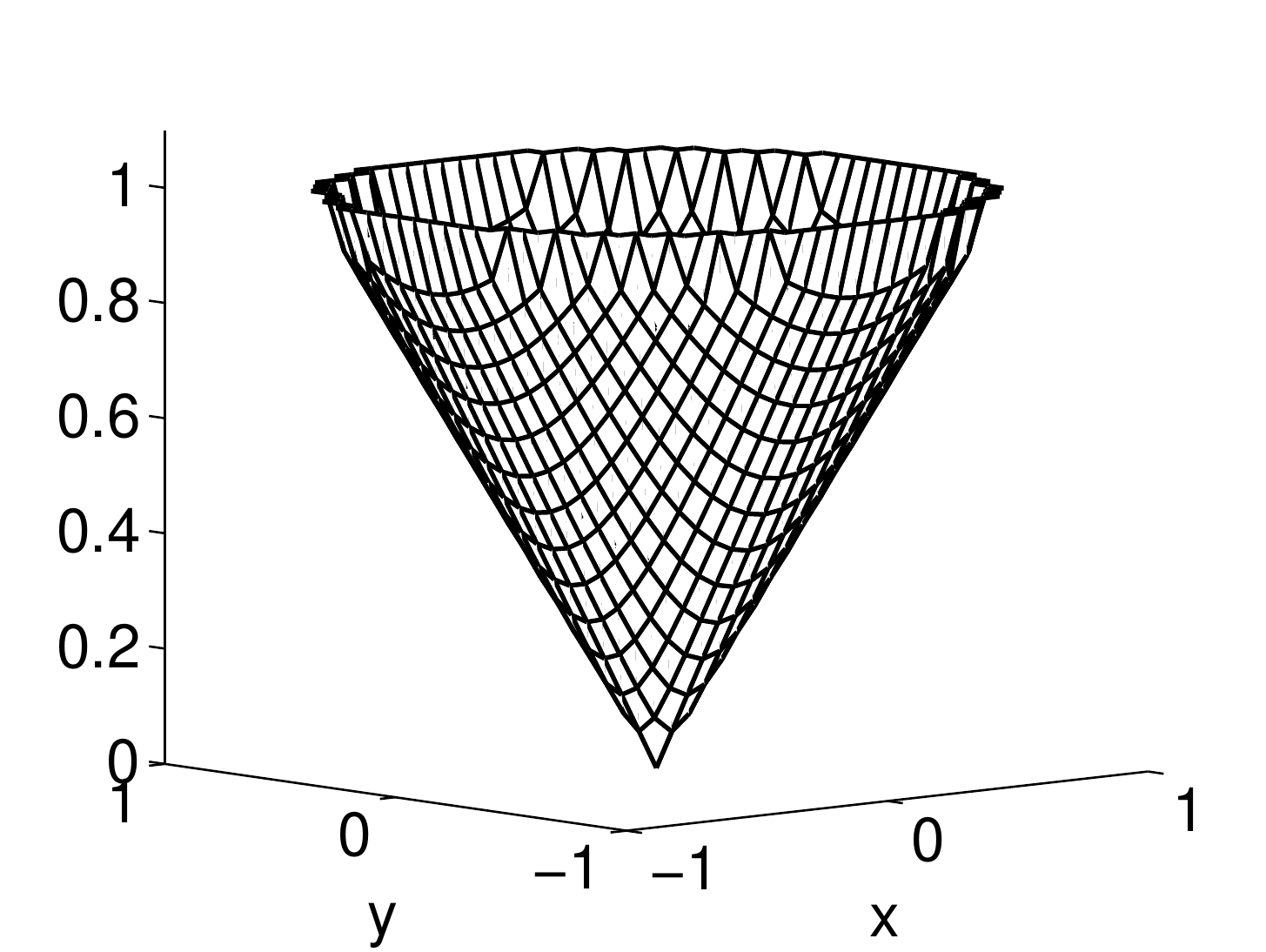}
       }
       
     \subfloat[]{
            \label{fb}
            \hspace{-.2\textwidth}
       \includegraphics[width=\ww\textwidth]{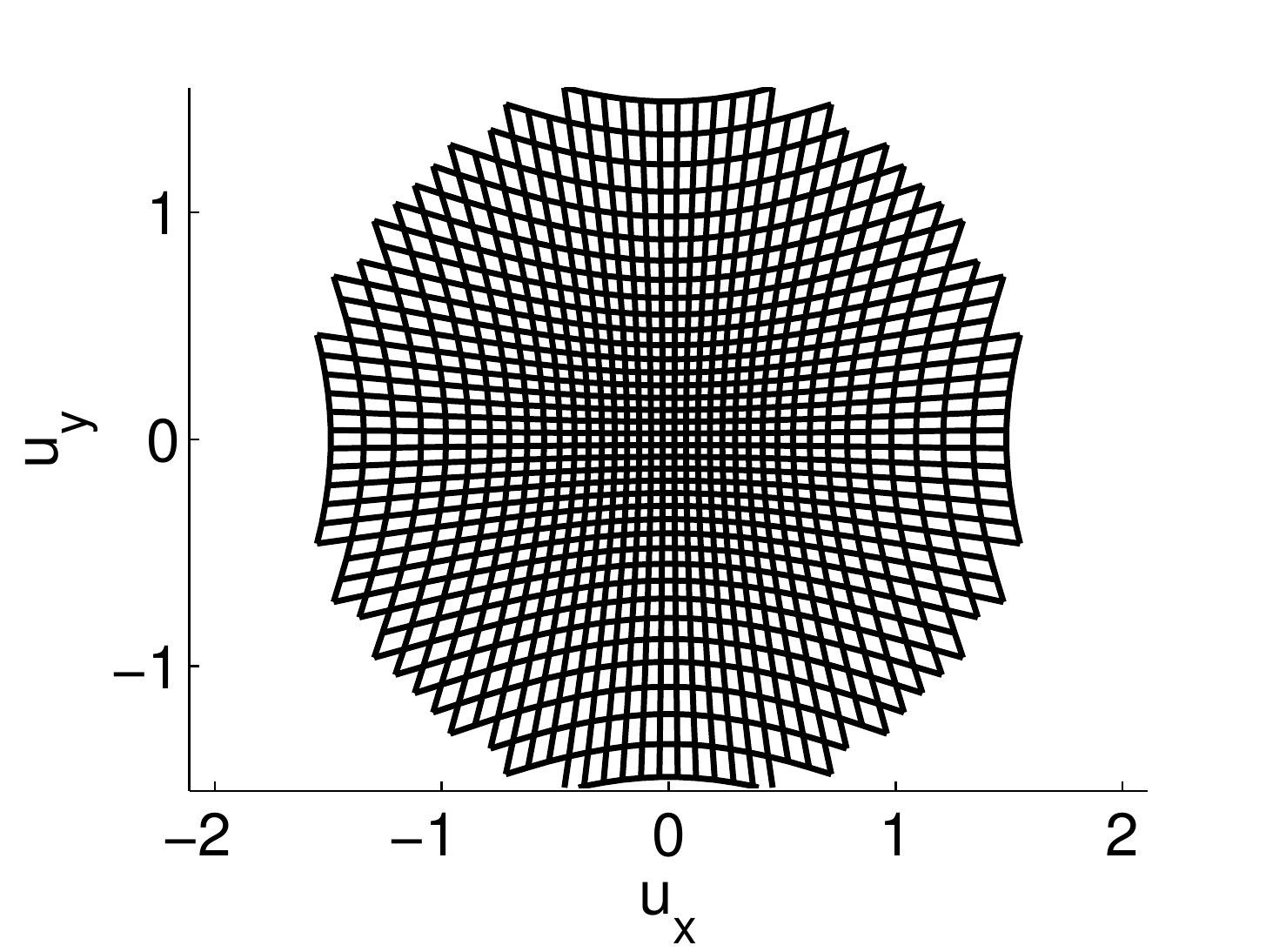} 
       \includegraphics[width=\ww\textwidth]{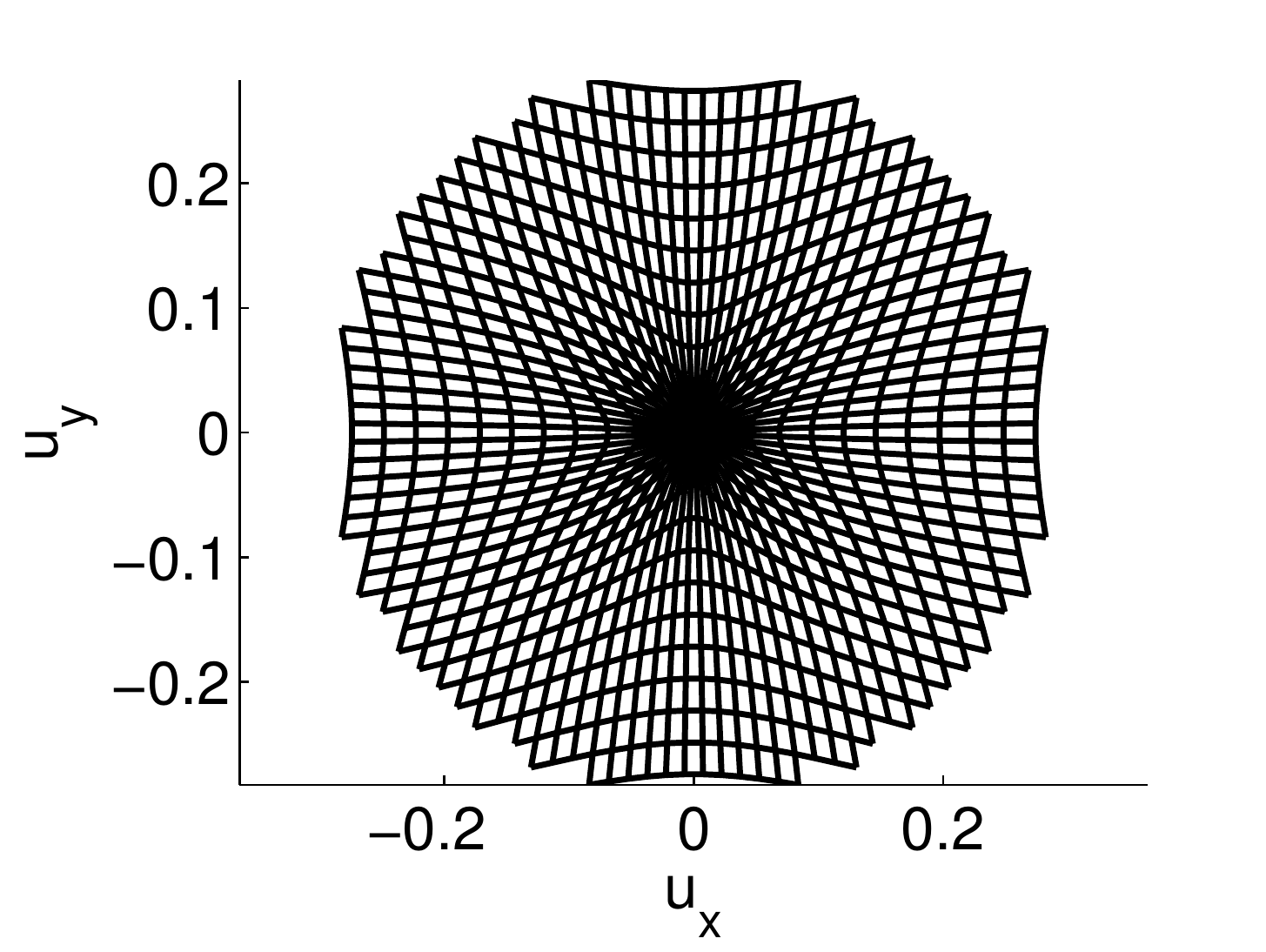}
       \includegraphics[width=\ww\textwidth]{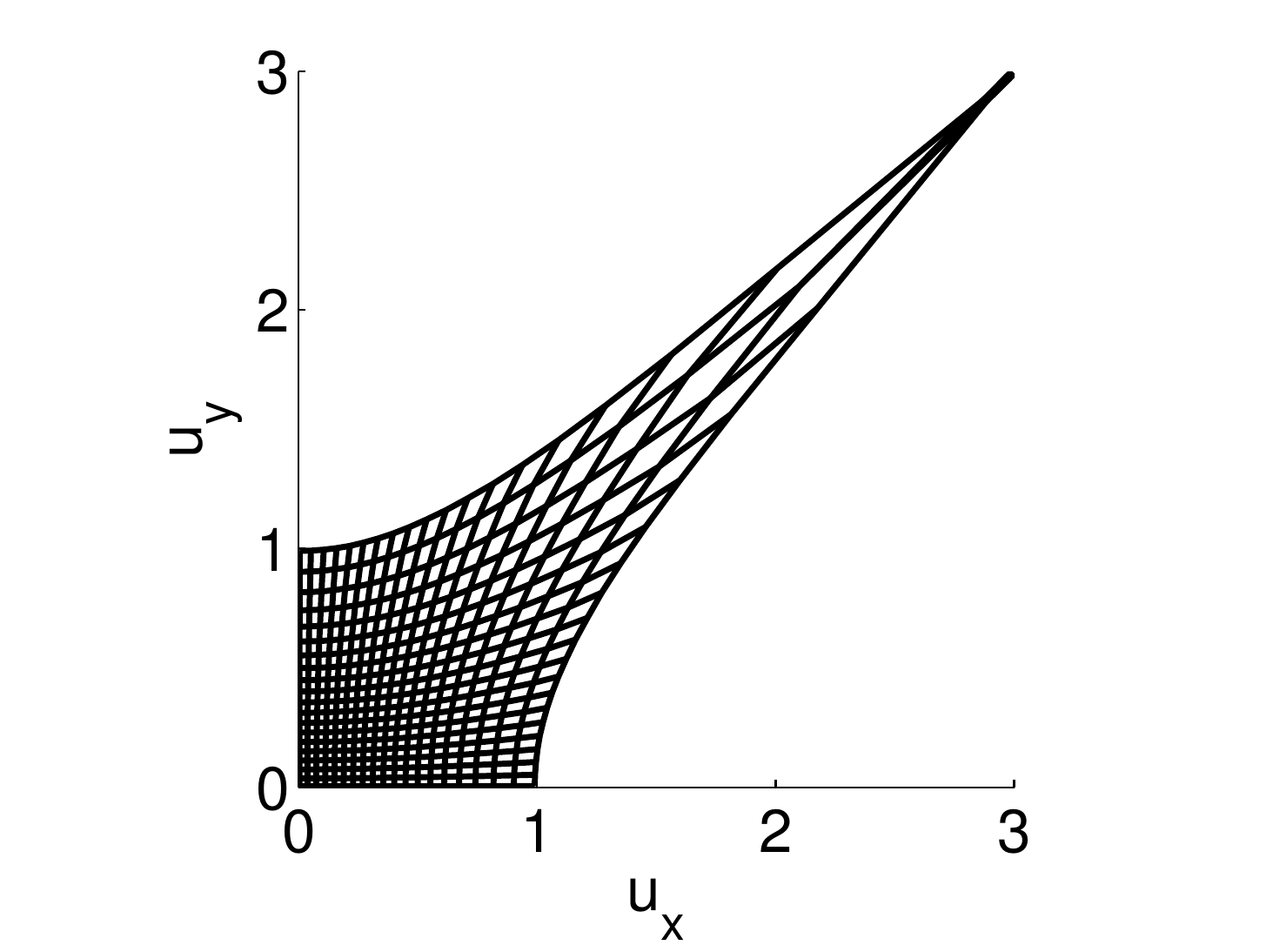}
        \hspace{\ww\textwidth}
}
\caption{
Top row: surface plot of the solutions, $u(x,y)$.  Bottow row: Mappings $\grad u(x,y)$ visualized by plotting the image of the constant $x$ and $y$ line segment in source domain.
The solutions correspond to (from left to right) 
$C^2$ solution,  $C^1$ solution, solution with blow-up, and cone solution (no mapping in this case, since it is singular). 
} 
\label{fig:solns}
\end{figure}

All computations are performed on the domain $[0,1]\times[0,1]$, which is discretized  on an $N\times N$ grid. We choose $\e(h,d\theta) = \sqrt{h} + d\theta/10$.  Computations are done using the 9, 17, and 33 point stencils.

\subsection{Discussion of computational results}

\emph{Accuracy}
We begin by comparing the maximum error in the filtered scheme to the error in the monotone scheme.
The accuracy is illustrated in \autoref{fig:accuracy} as well as in \autoref{table:errMA}.  As we expect, the filtered scheme results in \emph{improved accuracy}.  
The number of grid points was chosen so that in many examples the directional resolution error dominates the spatial resolution error.  This is the reason that the accuracy of the monotone scheme tapers off.  On the other hand, the hybrid and filtered schemes achieve better accuracy.  For the cone example, the theory does not ensure convergence, and in this case, the standard scheme is more accurate, but the computation time for the standard scheme is very long, and there are no guarantees for this scheme: the accuracy was much worse on other examples. 

Not surprisingly, the gains in accuracy are greatest on the smoothest solutions, when we can reasonably expect the more accurate scheme to be valid.  On more singular examples, the standard scheme may not be valid, and the filtered scheme can be forced to choose the less accurate monotone scheme at many points in the domain.  This is not a limitation of the scheme, though, since choosing the formally more accurate scheme everywhere can lead to instabilities or convergence to the wrong solution.

We also compare our results to the \emph{a priori} hybrid scheme, which, despite lacking a proof of convergence, was experimentally found to be efficient and more accurate than the monotone scheme~\cite{ObermanFroeseFast}.

\begin{figure}[htdp]
\subfloat{
        \hspace{-.2\textwidth}
       \includegraphics[width=\ww\textwidth]{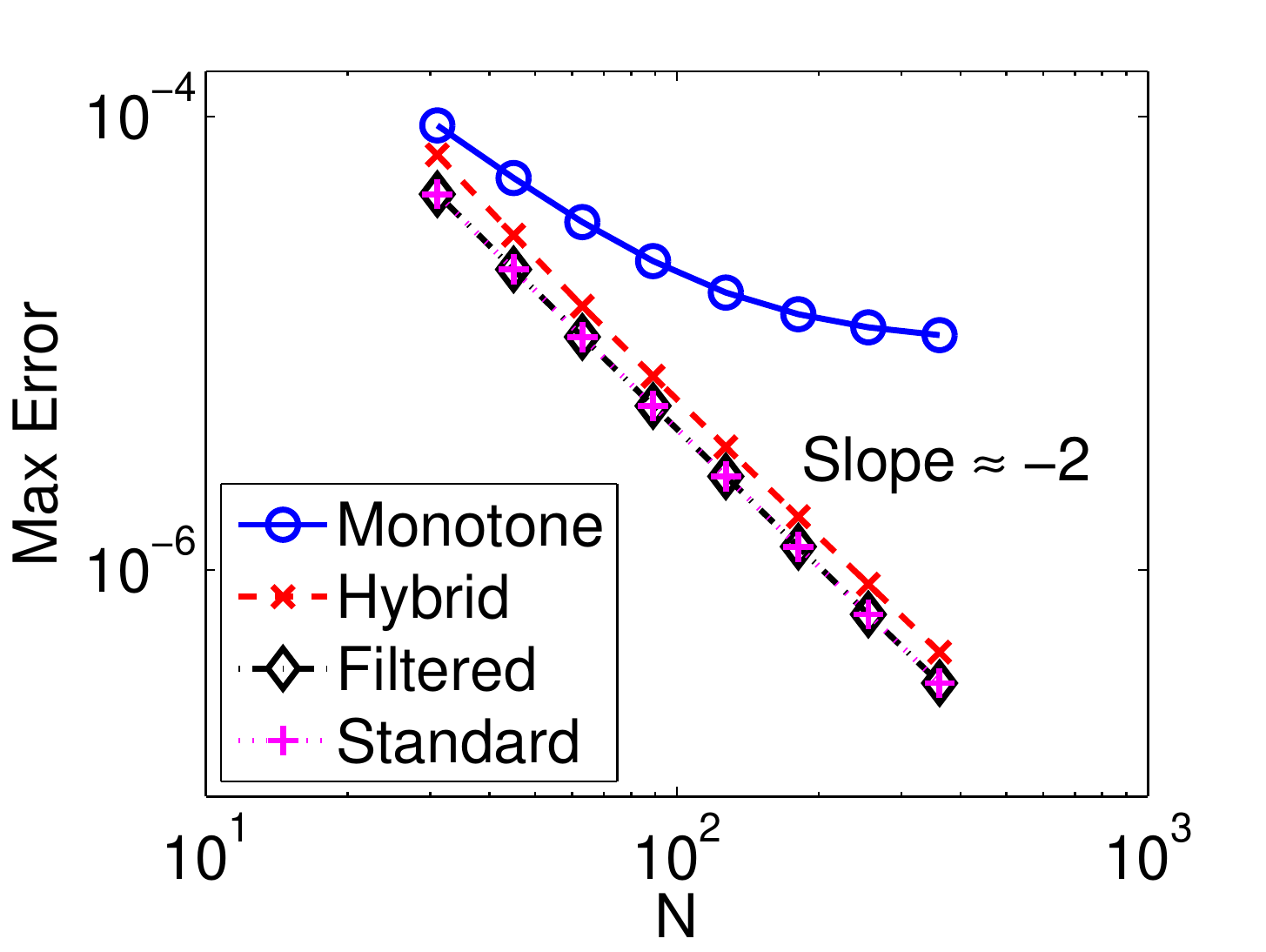}
       \includegraphics[width=\ww\textwidth]{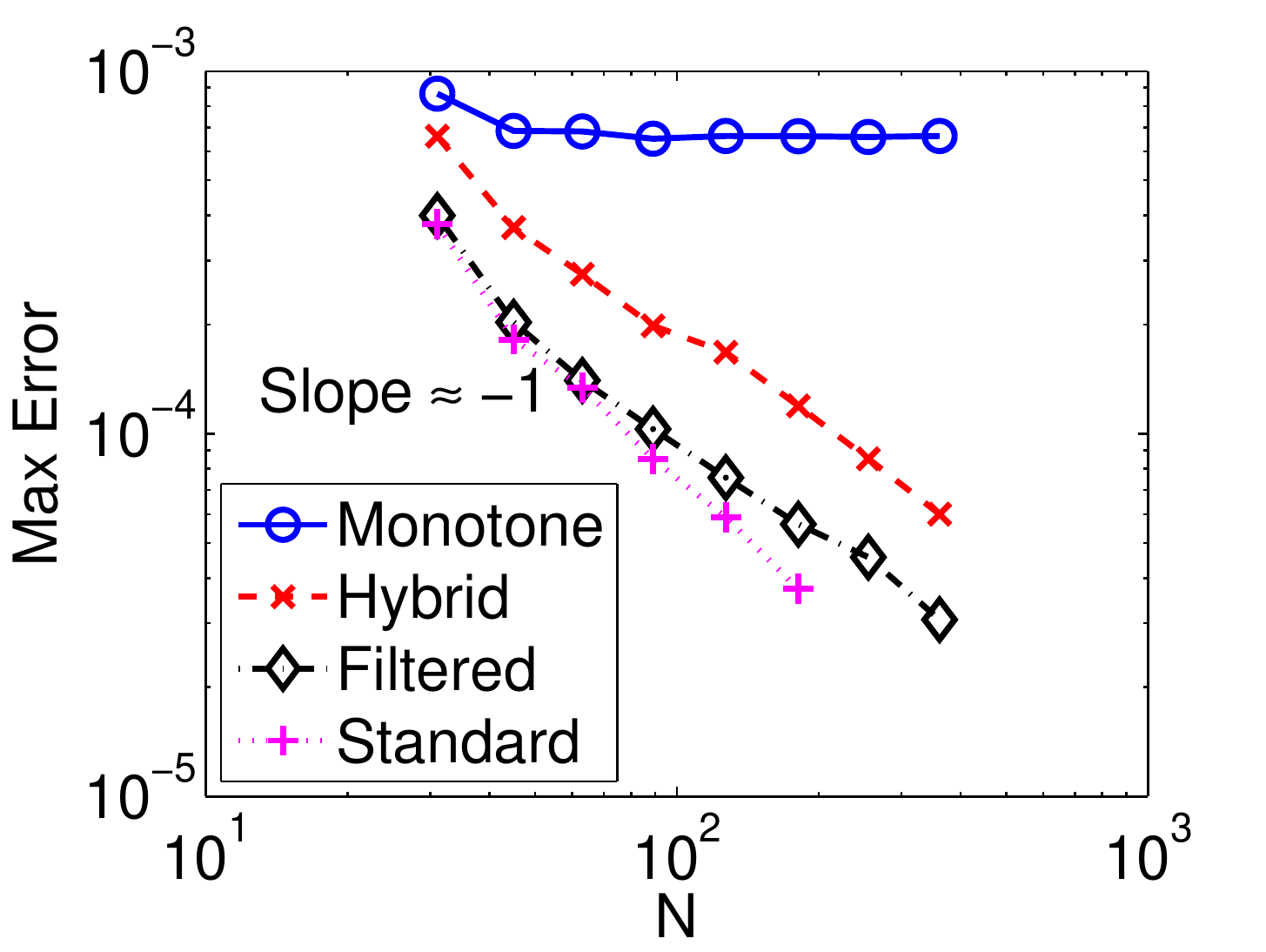}
       \includegraphics[width=\ww\textwidth]{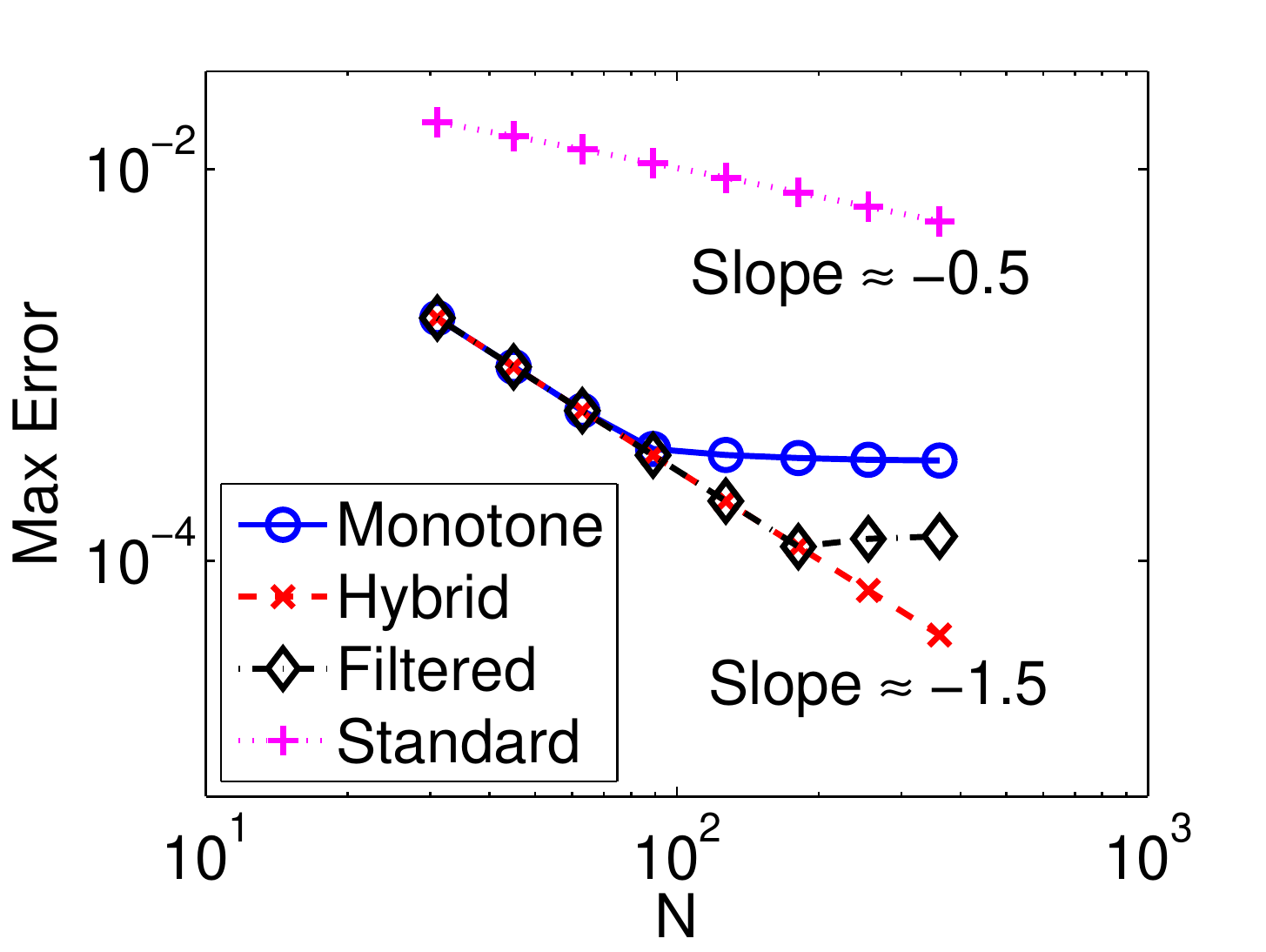}
       \includegraphics[width=\ww\textwidth]{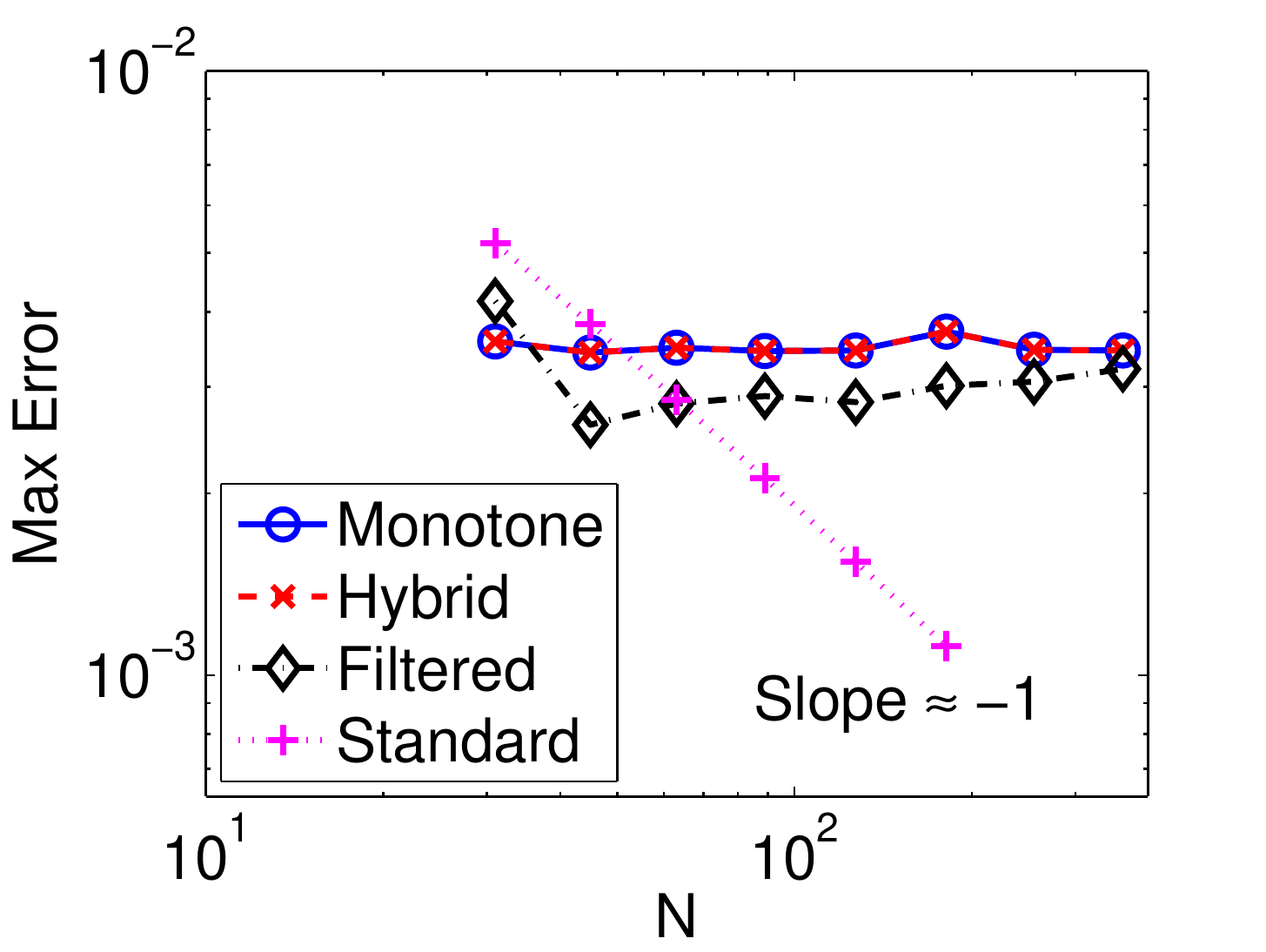}
       }
\caption{
Accuracy of the different schemes.  Solutions are ordered as in \autoref{fig:solns}. 
}
\label{fig:accuracy}
\end{figure}

\autoref{fig:filters} shows the weights of the monotone and accurate schemes determined by the filtered schemes.  For the smooth solution, the accurate method is always used.  Near singularities, the monotone method is selected.
For comparison, the corresponding weights for the hybrid scheme are also shown.  
The filtered schemes uses the accurate method at more locations than the hybrid discretization.   This occurs despite the fact that the hybrid scheme used known regularity results to choose the locations of the more accurate method, whereas the filtered scheme determines the discretization from the equation.

\emph{Computation speed.}
\autoref{table:hybridMA} presents the number of Newton iterations, computation time, and maximum error in each of the computed solutions for the hybrid and filtered schemes on a 17 point stencil.   As a benchmark, the error in the standard scheme is also presented; it has been shown previously that this scheme is much slower than the Newton solver for the hybrid discretization~\cite{ObermanFroeseFast}.  Even on this relatively narrow stencil, the accuracy of the filtered scheme is close to---and in some cases better than---the accuracy achieved using the slow, formally second-order standard scheme.  Overall, there is  \emph{no appreciable difference in accuracy}  between the results obtained using the hybrid and the filtered scheme, though the filtered scheme is slightly more accurate in most of the examples.  (With our particular choice of the parameter $\e(h,d\theta)$ it is slightly less accurate on the blow-up example.  Accuracy can be improved by allowing this parameter to increase).    Of course, the filtered scheme has one big advantage over the hybrid scheme because it comes with a convergence proof, which gives us confidence that the method will continue to perform correctly in other examples.

\begin{figure}[htdp]
       \subfloat[]{
              \label{fc}
       \includegraphics[width=\www\textwidth]{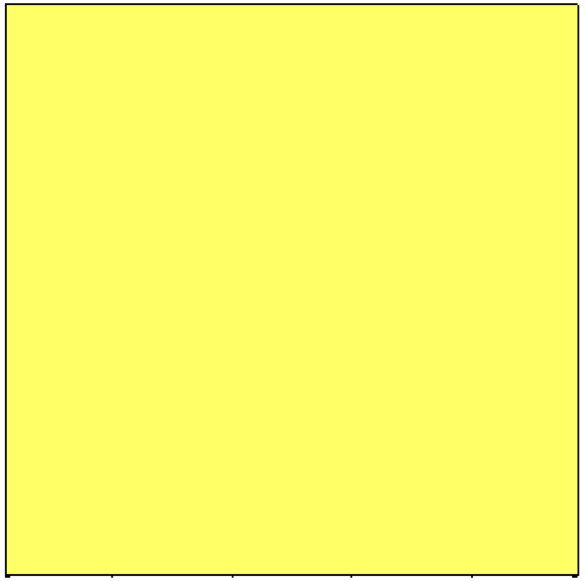}
       \includegraphics[width=\www\textwidth]{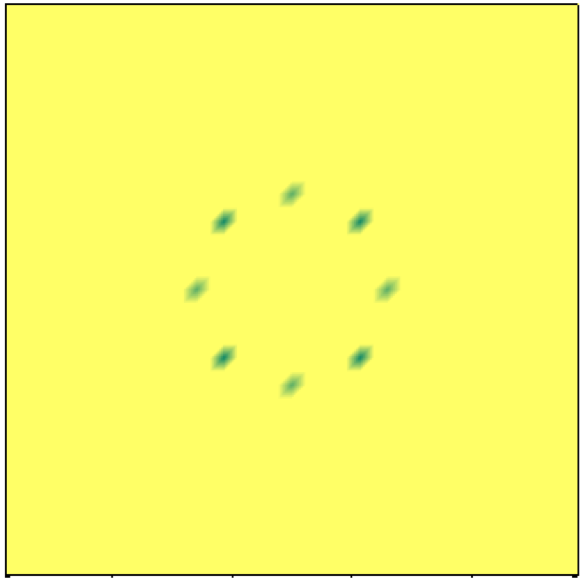}
       \includegraphics[width=\www\textwidth]{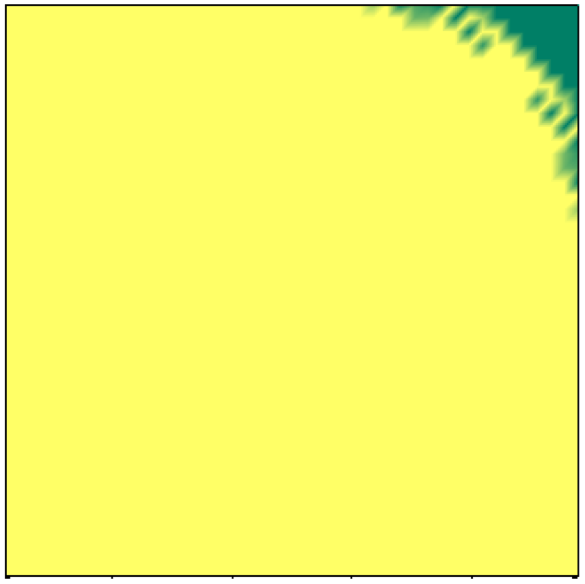}
       \includegraphics[width=\www\textwidth]{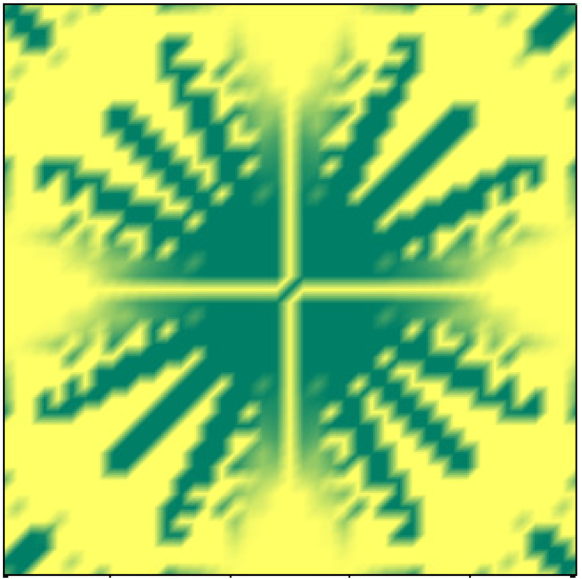}
}

       \subfloat[]{
              \label{fd}
       \includegraphics[width=\www\textwidth]{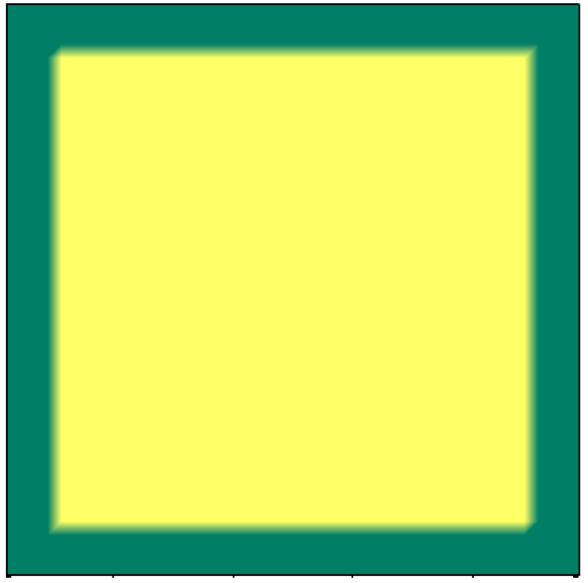}
       \includegraphics[width=\www\textwidth]{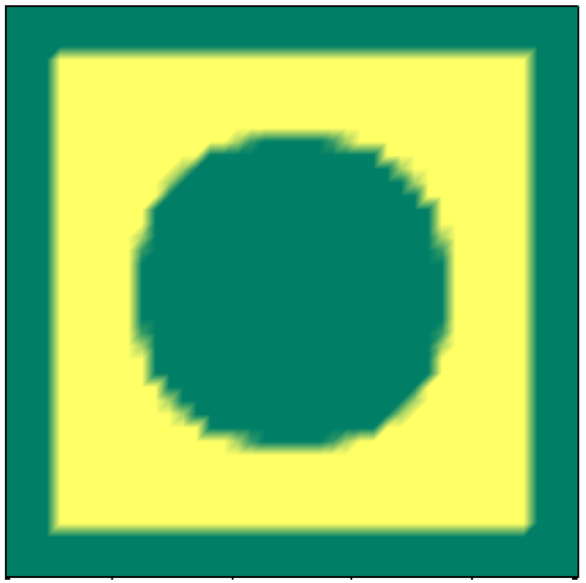}
       \includegraphics[width=\www\textwidth]{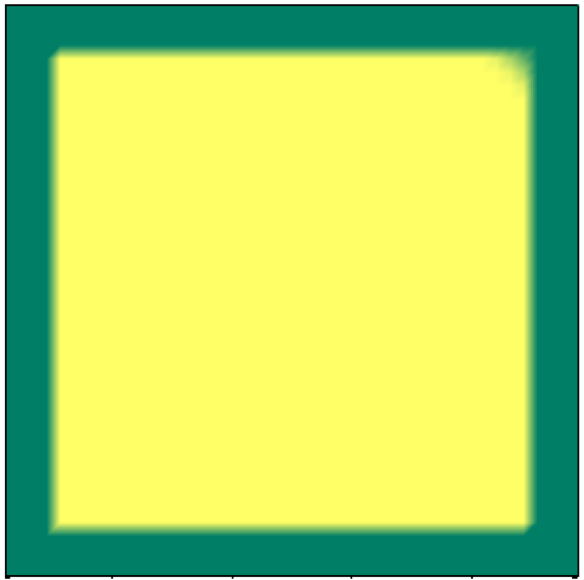}
       \includegraphics[width=\www\textwidth]{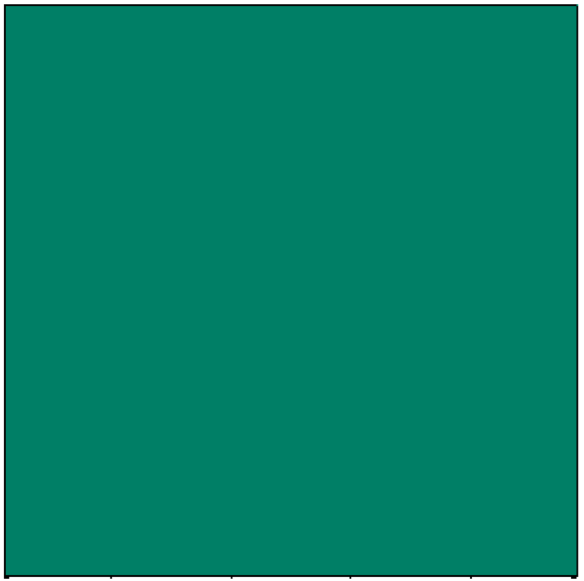}
}
\caption{
Illustration of the weights in the filtered scheme (\ref{fc}),
and the hybrid scheme (\ref{fd}).  The ordering is the same as in \autoref{fig:solns}.
In the second example top figure, green indicates a weight of  5\% on the monotone scheme.  In the third and fourth figures, green indicates the monotone scheme, while yellow indicates the accurate scheme.
} 
\label{fig:filters}
\end{figure}

\begin{table}[htdp]
\centering\footnotesize
\begin{tabular}{ccccccc}
\multicolumn{7}{c}{{Max Error,} $C^2$ Example}\\
N  & \multicolumn{2}{c}{9 Point}   &   \multicolumn{2}{c}{17 Point}  & \multicolumn{2}{c}{33 Point} \\
 &   Monotone & Filtered & Monotone & Filtered & Monotone & Filtered \\
\hline
31 & $9.45\ex{5}$ & $4.54\ex{5}$ & $9.12\ex{5}$& $4.54\ex{5}$ & $9.38\ex{5}$ &$4.54\ex{5}$ \\
63 & $4.91\ex{5}$ & $1.04\ex{5}$ & $3.42\ex{5}$& $1.04\ex{5}$ & $3.40\ex{5}$ &$1.04\ex{5}$  \\
127 & $3.79\ex{5}$& $0.26\ex{5}$ & $1.67\ex{5}$& $0.26\ex{5}$ & $1.39\ex{5}$ &$0.26\ex{5}$ \\
255 & $3.51\ex{5}$& $0.06\ex{5}$ & $1.17\ex{5}$& $0.06\ex{5}$ & $0.66\ex{5}$  &$0.06\ex{5}$ \\
361 & $3.48\ex{5}$& $0.03\ex{5}$ & $1.08\ex{5}$& $0.03\ex{5}$ & $0.51\ex{5}$  &$0.03\ex{5}$ \\
\hline\hline\\
\multicolumn{7}{c}{{Max Error,} $C^1$ Example}\\
N  & \multicolumn{2}{c}{9 Point}   &   \multicolumn{2}{c}{17 Point}  & \multicolumn{2}{c}{33 Point} \\
 &   Monotone & Filtered & Monotone & Filtered & Monotone & Filtered \\
\hline
31  & $21.54\ex{4}$& $3.73\ex{4}$ & $8.66\ex{4}$& $3.99\ex{4}$ & $6.39\ex{4}$ &$3.67\ex{4}$ \\
63  & $21.33\ex{4}$& $1.51\ex{4}$ & $6.82\ex{4}$& $1.40\ex{4}$ & $3.18\ex{4}$ &$1.46\ex{4}$ \\
127 & $21.55\ex{4}$& $0.92\ex{4}$ & $6.63\ex{4}$& $0.76\ex{4}$ & $2.49\ex{4}$ &$0.78\ex{4}$ \\
255 & $21.51\ex{4}$& $0.38\ex{4}$ & $6.58\ex{4}$& $0.46\ex{4}$ & $2.36\ex{4}$ &$0.37\ex{4}$  \\
361 & $21.53\ex{4}$& $0.23\ex{4}$ & $6.62\ex{4}$& $0.31\ex{4}$ & $2.37\ex{4}$ &$0.28\ex{4}$ \\\hline\hline\\
\multicolumn{7}{c}{{Max Error,} Example with Blow-up}\\
N  & \multicolumn{2}{c}{9 Point}   &   \multicolumn{2}{c}{17 Point}  & \multicolumn{2}{c}{33 Point} \\
 &   Monotone & Filtered & Monotone & Filtered & Monotone & Filtered \\
\hline
31  & $1.74\ex{3}$& $1.74\ex{3}$ & $1.74\ex{3}$& $1.74\ex{3}$ & $1.74\ex{3}$ &$1.74\ex{3}$ \\
63  & $0.86\ex{3}$& $0.59\ex{3}$ & $0.59\ex{3}$& $0.59\ex{3}$ & $0.59\ex{3}$ &$0.59\ex{3}$ \\
127 & $0.83\ex{3}$& $0.20\ex{3}$ & $0.35\ex{3}$& $0.20\ex{3}$ & $0.20\ex{3}$ &$0.20\ex{3}$ \\
255 & $0.83\ex{3}$& $0.15\ex{3}$ & $0.33\ex{3}$& $0.13\ex{3}$ & $0.16\ex{3}$ &$0.08\ex{3}$ \\
361 & $0.83\ex{3}$& $0.17\ex{3}$ & $0.33\ex{3}$& $0.13\ex{3}$ & $0.15\ex{3}$ &$0.08\ex{3}$  \\
\hline\hline\\
\multicolumn{7}{c}{{Max Error,} $C^{0,1}$ (Lipschitz) Example}\\
N  & \multicolumn{2}{c}{9 Point}   &   \multicolumn{2}{c}{17 Point}  & \multicolumn{2}{c}{33 Point} \\
 &   Monotone & Filtered & Monotone & Filtered & Monotone & Filtered \\
\hline
31  & $11.83\ex{3}$& $10.42\ex{3}$& $3.56\ex{3}$& $4.16\ex{3}$& $1.61\ex{3}$ & $1.23\ex{3}$  \\
63  & $11.10\ex{3}$& $11.56\ex{3}$& $3.49\ex{3}$& $2.82\ex{3}$& $1.65\ex{3}$ & $1.90\ex{3}$  \\
127 & $11.80\ex{3}$& $10.97\ex{3}$& $3.45\ex{3}$& $2.83\ex{3}$& $1.64\ex{3}$ & $1.11\ex{3}$  \\
255 & $10.47\ex{3}$& $11.03\ex{3}$& $3.46\ex{3}$& $3.06\ex{3}$& $1.64\ex{3}$ & $1.15\ex{3}$ \\
361 & $10.40\ex{3}$& $10.37\ex{3}$& $3.45\ex{3}$& $3.22\ex{3}$& $1.64\ex{3}$ & $1.12\ex{3}$  \\
\hline\hline \\
\end{tabular}
\caption{Accuracy of the monotone and filtered schemes.}
\label{table:errMA}
\end{table}

\begin{table}[htdp]
\centering\footnotesize
\begin{tabular}{cccccccc}
 & \multicolumn{7}{c}{$C^2$ Example}\\
 & \multicolumn{2}{c}{Iterations} & \multicolumn{2}{c}{CPU Time (s)} & \multicolumn{3}{c}{{Max Error}}\\
N  &  Hybrid &  Filtered&  Hybrid &  Filtered&  Hybrid &  Filtered & Standard\\
\hline
31 & 3 & 2 & 0.1& 0.1& $6.76\ex{5}$&  $4.54\ex{5}$ &  $4.54\ex{5}$\\
63 & 3 & 2 & 0.6& 0.5 & $1.46\ex{5}$&  $1.06\ex{5}$ &  $1.06\ex{5}$\\
127& 3 & 2 & 2.4& 2.0 & $0.35\ex{5}$&  $0.26\ex{5}$ &  $0.26\ex{5}$\\
255& 3 & 2 & 12.5& 9.9 & $0.09\ex{5}$&  $0.06\ex{5}$ &  $0.06\ex{5}$\\
361& 3 & 2 & 28.2& 22.5 & $0.04\ex{5}$&  $0.03\ex{5}$ &  $0.03\ex{5}$\\
\hline\hline\\
 & \multicolumn{7}{c}{$C^1$ Example}\\
 & \multicolumn{2}{c}{Iterations} & \multicolumn{2}{c}{CPU Time (s)} & \multicolumn{3}{c}{{Max Error}}\\
N  &  Hybrid &  Filtered&  Hybrid &  Filtered&  Hybrid &  Filtered & Standard\\
\hline
31  & 2 & 2 & 0.1& 0.1& $6.62\ex{4}$ & $3.99\ex{4}$ &  $3.78\ex{4}$\\
63  & 3 & 2 & 0.6& 0.5& $2.75\ex{4}$ & $1.40\ex{4}$ &  $1.34\ex{4}$\\
127 & 5 & 3 & 3.8& 2.6& $1.68\ex{4}$ & $0.76\ex{4}$ &  $0.59\ex{4}$ \\
255 & 5 & 3 & 19.1& 12.9& $0.85\ex{4}$ & $0.46\ex{4}$ & ---\\
361 & 5 & 6 & 48.3& 50.6& $0.60\ex{4}$ & $0.31\ex{4}$ & ---\\
\hline\hline\\
 & \multicolumn{7}{c}{Example with Blow-up}\\
 & \multicolumn{2}{c}{Iterations} & \multicolumn{2}{c}{CPU Time (s)} & \multicolumn{3}{c}{{Max Error}}\\
N  &  Hybrid &  Filtered&  Hybrid &  Filtered&  Hybrid &  Filtered & Standard\\
\hline
31  & 6 & 7 &0.3 & 0.3& $1.74\ex{3}$ & $1.74\ex{3}$ &  $17.38\ex{3}$\\
63  & 9 & 9 & 1.4& 1.5& $0.59\ex{3}$ & $0.59\ex{3}$ &  $12.62\ex{3}$\\
127 & 11 & 11 & 8.6& 8.4& $0.20\ex{3}$ & $0.20\ex{3}$ &  $9.04\ex{3}$\\
255 & 7 & 8 & 30.5& 32.4& $0.07\ex{3}$ & $0.13\ex{3}$ &  $6.43\ex{3}$\\
361 & 11 & 12 & 101.5& 108.7& $0.04\ex{3}$ & $0.13\ex{3}$ &  $5.42\ex{3}$\\
\hline\hline\\
 & \multicolumn{7}{c}{Lipschitz Example}\\
 & \multicolumn{2}{c}{Iterations} & \multicolumn{2}{c}{CPU Time (s)} & \multicolumn{3}{c}{{Max Error}}\\
N  &  Hybrid &  Filtered&  Hybrid &  Filtered&  Hybrid &  Filtered & Standard\\
\hline
31  & 6 & 7 & 0.2& 0.2& $3.57\ex{3}$& $4.16\ex{3}$ & $5.19\ex{3}$\\
63  & 6 & 9 & 1.0& 1.4& $3.49\ex{3}$& $2.82\ex{3}$ & $2.86\ex{3}$\\
127 & 8 & 8 & 6.6& 6.4& $3.45\ex{3}$& $2.83\ex{3}$ & $1.54\ex{3}$\\
255 & 9 & 10 & 47.1& 38.2& $3.46\ex{3}$& $3.06\ex{3}$ & ---\\
361 & 10 & 9 & 155.8& 81.7& $3.45\ex{3}$& $3.22\ex{3}$ & ---\\
\hline\hline\\
\end{tabular}
\caption{Newton iterations, computation time, and accuracy for the hybrid and filtered schemes on a 17 point stencil.}
\label{table:hybridMA}
\end{table}

\section{Conclusions}
We constructed and implemented a convergent, higher order accurate scheme for the \MA equation.
We extended the convergence theory of Barles and Souganidis by considering the more general class of nearly monotone schemes.  This new convergence proof applies in general to the class of nonlinear elliptic PDEs, and requires only an elliptic scheme as a foundation.

The  combined schemes are called  \emph{filtered} finite difference {approximations}.  The filtered scheme chooses between a convergent elliptic scheme and a more accurate scheme.  The selection principle is based on filtering the difference between the elliptic scheme and the more accurate scheme, reducing to the elliptic scheme when the difference is large, and the accurate scheme when the difference is small.

 The theory ensures, and computational results verify, that solutions of this scheme converge to the viscosity solution of the \MA equation even in the singular setting.  
 
 The accuracy of the filtered schemes was as good as the accuracy of the other methods.  Formal $\bO(h^2)$ accuracy was attained on the smooth example, while on singular examples the accuracy decreased as expected.
Newton's method resulted in a fast solver, with the same number of iterations (2-11, depending on the solution) as for the monotone method and the hybrid method.  This is much faster than other types of solvers, which typically break down on singular solutions, or require more iterations on larger problems.

For the convergence theory, we require $d\theta \to 0$, but in practice, using the filtered scheme means  the accuracy is not affected by $d\theta$, so we can use the narrow stencil scheme and still obtain accuracy corresponding to the regularity of the solution (e.g. $\bO(h^2)$ on smooth solutions, $\bO(h)$ on moderately singular solutions).
The only exception was the most singular example, which was not a viscosity solution, in any case.  

In summary, the new method combines all the advantages of our previous methods: speed of solution, accuracy, and a proof of convergence, while also allowing the use of a narrow stencil scheme in practice.

\bibliographystyle{alpha}
\bibliography{../biblio/FilteredSchemes}

\end{document}